\documentclass[oneside]{amsart}

\usepackage{amssymb}

\usepackage[%
papersize={8.5in,11in},%
left=1in,%
bottom=1.3in,%
textwidth=5.1in%
]{geometry}
\usepackage{enumitem}

\def\papertitle{%
{C}arath\'eodory number of homogeneous convex cones%
}%
\def\papertitleshort{%
{C}arath\'eodory number of homogeneous cones%
}%
\def\keywordsphrases{%
Homogeneous cone%
; spectrahedral representation%
; Carath\'eodory number%
}%

\def\authorenameshort{C.~B. Chua}

\def\authoreaddress{%
Division of Mathematical Sciences, %
School of Physical \& Mathematical Sciences, %
Nanyang Technological University, %
Singapore 637371, %
Singapore%
}%

\def\authoremail{cbchua@ntu.edu.sg}

\def\authorthanks{%
\authoreaddress(\texttt{\authoremail})%
}%

\theoremstyle{plain}
\newtheorem{theorem}{Theorem}[section]

\newtheorem{corollary}[theorem]{Corollary}
\newtheorem{proposition}[theorem]{Proposition}

\theoremstyle{definition}

\theoremstyle{remark}

\numberwithin{equation}{section}

\usepackage[
    linktocpage=true,
    colorlinks=true,
    citecolor=blue,
    pdftex, 
]{hyperref}

\hypersetup{pdfproducer={}}%
\hypersetup{pdfpagemode={}}%
\hypersetup{pdfstartview={XYZ null null 1}}%
\hypersetup{pdfstartview=FitH}%
\hypersetup{pdfauthor={Chek Beng Chua}}%
\hypersetup{pdfsubject={%
}}%
\hypersetup{pdfkeywords={\keywordsphrases}}%
\hypersetup{pdftitle={\papertitle}}%

\DeclareMathOperator{\trace}{trace}
\DeclareMathOperator{\proj}{proj}

\begin{document}

\title[\papertitleshort]{\papertitle}

\author[\authorenameshort]{\authorenameshort}

\thanks{\authorthanks}

\keywords{\keywordsphrases}

\subjclass[2020]{%
90C25; 52A20%
}%

\begin{abstract}
We study the Carath\'eodory number of homogeneous convex cones via their
spectrahedral representations. A characterization of homogeneous convex
cones whose ranks match their Carath\'eodory numbers is given. This
characterization is then used to show that a homogeneous convex cone is
selfdual if and only if its rank matches the Carath\'eodory numbers of both
its closure and its dual cone. It is further used to show that the only
sparse spectrahedral cones that are homogeneous convex cones are those
described by homogeneous chordal graphs.
\end{abstract}

\maketitle

\section{Introduction}
\label{section:introduction}

A homogeneous convex cone is a pointed\footnote{A cone is said to be
pointed if it does not contain any affine subspace.} convex cone in a
finite-dimensional vector space that is relatively
open\footnote{Traditionally, homogeneous convex cones, which arise from the
study of bounded homogeneous domains, are open cones, although it is common
in the optimization literature to define them as closed convex cones whose
relative interiors are homogeneous convex cones in the traditional sense.}
and whose linear automorphism group acts transitively on it. Homogeneous
convex cones form an important class of convex cones in convex conic
programming as it subsumes convex cones found in prominent convex programs
such as positive definite cones and symmetric cones. In this paper, we
study a geometric invariant of homogeneous convex cones known as the
Carath\'eodory number, followed by specializing our study to symmetric
cones and sparse positive definite cones.

The Carath\'eodory number at a point $x$ of a finite-dimensional closed
pointed convex cone $\bar{\mathsf{K}}$ is the minimum number of extreme
rays of $\bar{\mathsf{K}}$ needed to contain $x$ in their Minkowski sum; it
is denoted by $\kappa_{\bar{\mathsf{K}}}(x)$. The Carath\'eodory number of
$\bar{\mathsf{K}}$ is then the maximum Carath\'eodory number at all points
in $\bar{\mathsf{K}}$; it is denoted by $\kappa_{\bar{\mathsf{K}}}$. So,
every point in $\bar{\mathsf{K}}$ is in the Minkowski sum of
$\kappa_{\bar{\mathsf{K}}}$ extreme rays of $\bar{\mathsf{K}}$, with at
least one point requiring exactly $\kappa_{\bar{\mathsf{K}}}$ extreme rays.
Carath\'eodory's theorem implies that the Carath\'eodory number
$\kappa_{\bar{\mathsf{K}}}$ of a closed convex cone $\bar{\mathsf{K}}$ is
at most the dimension of the cone; i.e., the dimension of the linear span
of the cone. For details, see \cite[\S17]{Rock97}.

The Carath\'eodory number of the closure of a homogeneous convex cone was
proved by G\"{u}ler and Tun\c{c}el \cite{GulTun98} to be a lower bound on
an algebraic invariant of the homogeneous convex cone known as its
rank\footnote{According to Vinberg \cite[p.~377]{Vin63}, the rank of a
homogeneous convex cone can be formally defined as ``the dimension of its
maximal connected commutative group of automorphisms, consisting of
semisimple transformations with real eigenvalues.''}, by showing that the
rank of a homogeneous convex cone is the least possible barrier parameter
of self-concordant barriers used in interior-point methods for convex conic
programs over the homogeneous convex cone. One of the main contributions of
this paper is a characterization of homogeneous convex cones whose ranks
agree with the Carath\'eodory numbers of their closures.

In the same work \cite{GulTun98}, the authors further showed that when a
homogeneous convex cone is selfdual, which is then known as a symmetric
cone, its rank agrees with the Carath\'eodory number of its closure; see,
also, \cite[Theorem~8]{TruongTun04} and \cite[Theorem~13]{ItoLou17}.
Another main contribution of this paper uses our characterization of
homogeneous convex cones with rank-matching Carath\'eodory numbers to prove
that symmetric cones are the only homogeneous convex cones whose ranks
agree with the Carath\'eodory numbers of both their closures and their dual
cones. Along the way, we provide a new proof of the fact that the rank of a
symmetric cone agrees with the Carath\'eodory number of its closure.

Our study of the Carath\'eodory numbers of homogeneous convex cones and
their relation with the ranks of the cones requires an exposition of the
facial structure of the closures of homogeneous convex cones. Rather than
describing the faces of the closure of a homogeneous convex cone via
algebraic means (see \cite[\S4]{Chua06}), we find that this exposition is
more accessible when we represent the homogeneous convex cone as a
spectrahedral cone, which is a linear slice $\mathbb{S}^{n}_{++} \cap
\mathbb{L}$ of a cone of positive definite real matrices
$\mathbb{S}^{n}_{++}$ (or a linear slice $\mathbb{S}^{n}_{+} \cap
\mathbb{L}$ of a cone of positive semidefinite real matrices
$\mathbb{S}^{n}_{+}$ when representing its closure). Indeed, every
homogeneous convex cone has a spectrahedral representation; see
\cite{Chua02,Ishi15,Rot66}. Through these spectrahedral representations, we
discover explicit spectrahedral representations of the proper faces of the
closures of homogeneous convex cones and their dual cones, from which we
uncover a description of the rank as a geometric invariant and recover the
fact that the rank of a homogeneous convex cone is an upper bound on the
Carath\'eodory number of its closure.

In the case when the spectrahedral representation $\mathbb{S}^{n}_{++} \cap
\mathbb{L}$ simply describes a sparsity pattern on positive definite
matrices of order $n$ (i.e., the linear subspace $\mathbb{L}$ describes a
sparsity pattern), we show that this bound is tight (i.e., the rank matches
the Carath\'eodory number). Moreover, we use this tightness to characterize
all sparsity patterns whose sets of sparse positive definite matrices
described by them are homogeneous convex cones. As far as we know, this is
the first geometric proof of a known result established by Ishi
\cite{Ishi13} via a purely algebraic proof.

\subsection{Organization}
\label{section:organization}

This paper is organized as follows. We begin the next section by giving,
for any homogeneous convex cone, the explicit description of a
spectrahedral representation and a simply transitive group of linear
automorphisms represented by linear automorphisms of the representing
positive definite cone, as provided by Ishi in \cite{Ishi15}. We then do
the same for the dual cone. In Section~\ref{section:facialstructure}, we
provide explicit spectrahedral representations of faces of closures of
homogeneous convex cones and their dual cones, which allow us to describe
the rank as a geometric invariant and establish it as an upper bound on the
Carath\'eodory number of its closure. We then use these explicit
descriptions of the faces in Section~\ref{section:Caratheodorynumbers} to
give a characterization of homogeneous convex cones whose ranks match the
Carath\'eodory numbers of their closures. A similar characterization is
also provided for the dual cones in the same section. Finally, in
Section~\ref{section:applications}, we apply these characterizations to i)
show that a homogeneous convex cone is selfdual if and only if its rank
matches the Carath\'eodory numbers of its closure and its dual cone; and
ii) show that when a sparsity pattern on positive definite matrices
produces a homogeneous convex cone, the sparsity pattern must be described
by a homogeneous chordal graph.

\subsection{Notations and conventions}
\label{section:notations}

We shall use the blackboard bold font (e.g., $\mathbb{R}$, $\mathbb{V}$,
$\mathbb{T}$) to denote (finite-dimensional) real vector spaces, a sans
serif font (e.g., $\mathsf{K}$, $\mathsf{T}$, $\mathsf{F}$) to denote sets,
and a calligraphy font (e.g., $\mathcal{T}$, $\mathcal{L}$, $\mathcal{I}$)
to denote linear transformations between real vector spaces. The notations
$\mathbb{R}$, $\mathbb{R}^{n}$, $\mathbb{R}^{m \times n}$, and
$\mathbb{S}^{n}$ denote the real vector spaces of, respectively, the real
numbers, the real $n$-vectors, the real $m$-by-$n$ matrices, and the real
symmetric matrices of order $n$. We denote by $I_{n}$ the identity matrix
in $\mathbb{S}^{n}$ and by $O_{m \times n}$ the zero matrix in
$\mathbb{R}^{m \times n}$. We use $\mathbb{S}^{n}_{++}$ and
$\mathbb{S}^{n}_{+}$ to denote the cones of, respectively, positive
definite and positive semidefinite matrices in $\mathbb{S}^{n}$, and may
use $\mathbb{R}_{++}$ and $\mathbb{R}_+$ instead when $n = 1$. Unless
otherwise specified, we always use the trace inner product $\left\langle
\cdot \middle| \cdot \right\rangle_{\mathrm{F}} : (X, Y) \in \mathbb{L}
\times \mathbb{L} \mapsto \trace\left( X^{\top} Y \right)$ on every linear
subspace $\mathbb{L}$ of real matrices, where $X^{\top}$ denotes the
transpose of the matrix $X$, together with the induced Frobenius norm
$\left\lVert \cdot\right\rVert_{F} : X \in \mathbb{L} \mapsto
\sqrt{\left\langle X \middle| X \right\rangle_{\mathrm{F}}}$, and identify
its dual vector space with itself via the trace inner product. For a
relatively open cone $\mathsf{K}$ in $\mathbb{S}^{n}$, we denote by
$\bar{\mathsf{K}}^{*}$ its dual cone $\left\{Y \in \mathbb{S}^{n}\
\middle|\ \left\langle X \middle| Y \right\rangle_{\mathrm{F}} \geq 0 \
\forall X \in \mathsf{K}\right\}$, which is the closed convex cone of
nonnegative linear functions on $\mathsf{K}$, and reserve the notation
$\mathsf{K}^{*}$ for the relative interior of the dual cone. The orthogonal
projection onto a linear subspace $\mathbb{L}$ of real matrices is denoted
by $\proj_{\mathbb{L}}$. The adjoint transformation of a linear
transformation $\mathcal{T} : \mathbb{V} \to \mathbb{W}$ between linear
subspaces of real matrices is denoted by $\mathcal{T}^{*} : \mathbb{W} \to
\mathbb{V}$, and it is defined by
\[
\left\langle \mathcal{T} X \middle| Y \right\rangle_{\mathrm{F}} =
\left\langle X \middle| \mathcal{T}^{*} Y \right\rangle_{\mathrm{F}} \quad
\forall (X, Y) \in \mathbb{V} \times \mathbb{W}.
\]

\section{Spectrahedral representations of homogeneous convex cones}
\label{section:spectrahedralrepresentations}

In \cite{Ishi15}, Ishi gave, for a given homogeneous convex cone of rank
$r$, an explicit description of a spectrahedral representation $\mathsf{K}
= \mathbb{S}^{n}_{++} \cap \mathbb{V}$ of order $n$, where $\mathbb{V}$ is
a linear subspace of real symmetric block matrices of the form
\begin{equation}\label{eq:spaceofsymmblockmatrix}%
\mathbb{V} := \left\{
\begin{pmatrix}
X_{11} & X_{12} & \cdots & X_{1r}
\\
X_{12}^{\top} & X_{22} & \ddots & \vdots
\\
\vdots & \ddots & \ddots & X_{r-1,r}
\\
X_{1r}^{\top} & \cdots & X_{r-1,r}^{\top} & X_{rr}
\end{pmatrix}
\ \middle|\
\begin{aligned}
&X_{ii} \in \mathbb{R} I_{n_{i}} \ (1 \leq i \leq r)
\\
&X_{ij} \in \mathbb{V}_{ij} \ (1 \leq i < j \leq r)
\end{aligned}
\right\},
\end{equation}
for some natural numbers $n_{1}, \dotsc, n_{r}$ with $n_{1} + \dotsb +
n_{r} = n$ and some linear subspaces $\mathbb{V}_{ij} \subseteq
\mathbb{R}^{n_{i} \times n_{j}}$ $(1 \leq i < j \leq r)$. We shall call
this a \emph{block matrix spectrahedral cone} and reserve the notation
$\mathbb{S}^{n}_{++} \cap \mathbb{V}$ for block matrix spectrahedral cones,
while using $\mathbb{S}^{n}_{++} \cap \mathbb{L}$ for general spectrahedral
cones. For convenience, we shall denote by $\mathbb{V}_{ii}$ the linear
subspace $\mathbb{R} I_{n_{i}} \subset \mathbb{R}^{n_{i} \times n_{i}}$.

\begin{theorem}[Theorem~2 and Proposition~2 of \cite{Ishi15}]
\label{theorem:Ishirepresentation}%
Every homogeneous convex cone is linearly isomorphic to a block matrix
spectrahedral cone $\mathbb{S}^{n}_{++} \cap \mathbb{V}$ where the linear
subspaces $\mathbb{V}_{ij}$ $(1 \leq i \leq j \leq r)$ satisfy the
properties
\begin{enumerate}[label={[V\arabic*]},ref={[V\arabic*]}]
\item\label{axiom:triangularproductblock}%
$A_{ij} B_{jk} \in \mathbb{V}_{ik}$ for all $(A_{ij}, B_{jk}) \in
\mathbb{V}_{ij} \times \mathbb{V}_{jk}$ with $1 \leq i \leq j \leq k \leq
r$,

\item\label{axiom:mixedtriangularproductblock}%
$A_{ij}^{\top} B_{ik} \in \mathbb{V}_{jk}$ for all $(A_{ij}, B_{ik}) \in
\mathbb{V}_{ij} \times \mathbb{V}_{ik}$ with $1 \leq i \leq j < k \leq r$,
and

\item\label{axiom:selfproductblock}%
$A_{ij}^{\top} A_{ij} \in \mathbb{V}_{jj}$ for all $A_{ij} \in
\mathbb{V}_{ij}$ with $1 \leq i \leq j \leq r$.

\end{enumerate}
\end{theorem}

We shall call the block matrix spectrahedral cone $\mathbb{S}^{n}_{++} \cap
\mathbb{V}$ in Theorem~\ref{theorem:Ishirepresentation} an \emph{Ishi
spectrahedral representation} of the homogeneous convex cone.

\subsection{A simply transitive linear automorphism group}
\label{section:simplytransitivegroup}

The three properties \ref{axiom:triangularproductblock},
\ref{axiom:mixedtriangularproductblock}, and \ref{axiom:selfproductblock}
of the linear subspaces $\mathbb{V}_{ij}$ in the Ishi spectrahedral
representation $\mathbb{S}^{n}_{++} \cap \mathbb{V}$ show that the linear
subspace of upper triangular matrices $T \in \mathbb{R}^{n \times n}$
satisfying $T + T^{\top} \in \mathbb{V}$,
\[
\mathbb{T} := \left\{
\begin{pmatrix}
T_{11} & T_{12} & \cdots & T_{1r}
\\
O_{n_{1} \times n_{2}}^{\top} & T_{22} & \cdots & T_{2r}
\\
\vdots & \vdots & \ddots & \vdots
\\
O_{n_{1} \times n_{r}}^{\top} & O_{n_{2} \times n_{r}}^{\top} & \cdots &
T_{rr}
\end{pmatrix}
\ \middle|\
\begin{aligned}
&T_{ii} \in \mathbb{V}_{ii} = \mathbb{R} I_{n_{i}} \ (1 \leq i \leq r)
\\
&T_{ij} \in \mathbb{V}_{ij} \ (1 \leq i < j \leq r)
\end{aligned}
\right\},
\]
plays a special role. Indeed, for a general block matrix spectrahedral cone
$\mathbb{S}^{n}_{++} \cap \mathbb{V}$, the property
\ref{axiom:triangularproductblock} is equivalent to
\begin{enumerate}[label={[U\arabic*]},ref={[U\arabic*]}]
\item\label{axiom:triangularproduct}%
$T U \in \mathbb{T}$ for all $T, U \in \mathbb{T}$,

\end{enumerate}
while the two properties \ref{axiom:mixedtriangularproductblock} and
\ref{axiom:selfproductblock} are, together, equivalent to
\begin{enumerate}[label={[U\arabic*]},ref={[U\arabic*]}]
\setcounter{enumi}{1}%
\item\label{axiom:selfproduct}%
$T^{\top}  T \in \mathbb{V}$ for all $T \in \mathbb{T}$.

\end{enumerate}
This property \ref{axiom:selfproduct} is generally weaker than
\begin{enumerate}[label={[U\arabic*$'$]},ref={[U\arabic*$'$]}]
\setcounter{enumi}{1}%
\item\label{axiom:selfproductX}%
$T^{\top} X T \in \mathbb{V}$ for all $X \in \mathbb{V}$ and all $T \in
\mathbb{T}$,

\end{enumerate}
but would be equivalent to it under the property
\ref{axiom:triangularproduct}. This is because the polarization of
\ref{axiom:selfproduct} is
\[
T^{\top}  U + U^{\top} T \in \mathbb{V} \text{ for all } T, U \in
\mathbb{T},
\]
which, together with \ref{axiom:triangularproduct}, implies that for each
$X \in \mathbb{V}$ and each $T \in \mathbb{T}$,
\[
T^{\top} X T = T^{\top} \hat{X} T + T^{\top} \hat{X}^{\top} T \in
\mathbb{V}
\]
when $\hat{X}$ is the block-triangular matrix in $\mathbb{T}$ satisfying $X
= \hat{X} + \hat{X}^{\top}$.

The two properties \ref{axiom:triangularproduct} and
\ref{axiom:selfproductX} are, in turn, respectively equivalent to the
seemingly weaker properties
\begin{enumerate}[label={[U\arabic*\textsuperscript{+}]},
ref={[U\arabic*\textsuperscript{+}]}]
\item\label{axiom:triangularproductpositivediagonal}%
$T U \in \mathbb{T}$ for all $T, U \in \mathbb{T}_{++}$, and

\item\label{axiom:selfproductXpositivediagonal}%
$T^{\top} X T \in \mathbb{V}$ for all $X \in \mathbb{V}$ and all $T \in
\mathbb{T}_{++}$,

\end{enumerate}
where $\mathbb{T}_{++}$ denotes the subset of matrices in $\mathbb{T}$ with
positive diagonals, because for every $T, U \in \mathbb{T}$, there is some
$\alpha I_{n}$ with $\alpha \geq 0$ for which $T + \alpha I_{n}, U + \alpha
I_{n} \in \mathbb{T}_{++}$, so that the property
\ref{axiom:triangularproductpositivediagonal} implies that
\[
T U = (T + \alpha I_{n}) (U + \alpha I_{n}) - \alpha T - \alpha U -
\alpha^2 I_{n} \in \mathbb{T},
\]
and the property \ref{axiom:selfproductXpositivediagonal} implies that, for
each $X \in \mathbb{V}$,
\[
T^{\top} X T = 2 (T + \alpha I_{n})^{\top} X (T + \alpha I_{n}) - (T + 2
\alpha I_{n})^{\top} X (T + 2 \alpha I_{n}) + 2 \alpha^2 X \in \mathbb{V}.
\]

The two properties \ref{axiom:triangularproductpositivediagonal} and
\ref{axiom:selfproductXpositivediagonal} on $\mathbb{T}$ further show that
the set of linear transformations
\[
\mathsf{T}_{\mathsf{K}} := \left\{\mathcal{T}_{T} : X \in \mathbb{V}
\mapsto T^{\top} X T\ \middle|\ T \in \mathbb{T}_{++}\right\}
\]
plays a special role. The property
\ref{axiom:triangularproductpositivediagonal} implies that
$\mathsf{T}_{\mathsf{K}}$ is closed under composition; i.e.,
$\mathsf{T}_{\mathsf{K}}$ is a linear transformation monoid. Conversely,
when $\mathsf{T}_{\mathsf{K}}$ is closed under composition, for every $T, U
\in \mathbb{T}_{++}$, the composition $\mathcal{T}_{U} \mathcal{T}_{T} =
\mathcal{T}_{T U}$ coincides with some $\mathcal{T}_{\hat{T}} \in
\mathsf{T}_{\mathsf{K}}$ for some matrix $\hat{T} \in \mathbb{T}_{++}$, and
hence, $T U = \hat{T} \in \mathbb{T}$ is the unique Cholesky factor of the
positive definite matrix $(T U)^{\top} (T U) = \mathcal{T}_{T U}(I_{n}) =
\mathcal{T}_{\hat{T}}(I_{n}) = \hat{T}^{\top} \hat{T}$. The property
\ref{axiom:selfproductXpositivediagonal} implies that the linear
transformations in $\mathsf{T}_{\mathsf{K}}$ are linear endomorphisms of
$\mathsf{K}$. Conversely, when the linear transformations in
$\mathsf{T}_{\mathsf{K}}$ are linear endomorphisms of $\mathsf{K}$, each
linear endomorphism $\mathcal{T}_{T} \in \mathsf{T}_{\mathsf{K}}$ maps
every matrix $X$ in $\mathsf{K} = \mathbb{S}^{n}_{++} \cap \mathbb{V}$ to a
matrix $\mathcal{T}_{T}(X) = T^{\top} X T$ in $\mathsf{K}$, and hence maps
every matrix $X$ in the linear span $\mathbb{V}$ of $\mathsf{K}$ to a
matrix $\mathcal{T}_{T}(X) = T^{\top} X T$ in the linear span $\mathbb{V}$.

Furthermore, the property \ref{axiom:triangularproductpositivediagonal}
implies that every matrix $T \in \mathbb{T}_{++}$ has its inverse matrix
$T^{-1}$ in $\mathbb{T}$, because
\[
T^{-1} = D^{-1} (I - N D^{-1} + \dotsb + (-1)^{r-1} (N D^{-1})^{r-1}),
\]
where $T = D + N$ is the decomposition of $T$ into the sum of a diagonal
matrix $D \in \mathbb{T}$ and a nilpotent matrix $N \in \mathbb{T}$.
Therefore, the property \ref{axiom:triangularproductpositivediagonal} is,
in fact, equivalent to $\mathsf{T}_{\mathsf{K}}$ being a linear
transformation group, so that, together with
\ref{axiom:selfproductXpositivediagonal}, it is equivalent to
$\mathsf{T}_{\mathsf{K}}$ being a group of linear automorphisms of
$\mathsf{K}$.

Moreover, the two properties \ref{axiom:triangularproductpositivediagonal}
and \ref{axiom:selfproductXpositivediagonal} imply that every matrix $X$ in
the Ishi spectrahedral representation $\mathsf{K} = \mathbb{S}^{n}_{++}
\cap \mathbb{V}$ is the image $\mathcal{T}_{T}(I_{n}) = T^{\top} T$ of
$I_{n}$ under $\mathcal{T}_{T} \in \mathsf{T}_{\mathsf{K}}$ with $T$ being
the product $X(r)^{1/2} T(r-1)^{-1} \dotsb T(1)^{-1}$ of the
block-triangular matrices in the recursion
\[
X(k+1) = T(k)^{\top} X(k) T(k) \in \mathbb{V} \quad \text{for } k = 1,
\dotsc, r-1,
\]
where $X(1)$ is the matrix $X$ and $T(k) \in \mathbb{T}$ is the triangular
matrix with ones on the diagonal, $T(k)_{kj} = -X(k)_{kk}^{-1} X(k)_{kj}$
for $j \in \left\{k+1, \dotsc, r\right\}$, and zero blocks everywhere else.
Therefore, the two properties \ref{axiom:triangularproductpositivediagonal}
and \ref{axiom:selfproductXpositivediagonal} (or, equivalently, the
properties \ref{axiom:triangularproductblock},
\ref{axiom:mixedtriangularproductblock}, and \ref{axiom:selfproductblock})
are, together, equivalent to the property that
\begin{enumerate}[label={[T]},ref={[T]}]
\item\label{axiom:automorphismgroup}%
$\mathsf{T}_{\mathsf{K}}$ is a simply transitive group of linear
automorphisms of $\mathsf{K}$.

\end{enumerate}
This shows that the Ishi spectrahedral representation $\mathsf{K} =
\mathbb{S}^{n}_{++} \cap \mathbb{V}$ is a homogeneous convex cone; see,
also, \cite[Proposition~2 and Theorem~3]{Ishi15}. Of course, this means
that $\mathsf{K}$ is also the orbit of $I_{n} \in \mathsf{K}$ under the
action of the transitive group $\mathsf{T}_{\mathsf{K}}$; i.e.,
\[
\mathsf{K} = \left\{\mathcal{T}_{T} I_{n} = T^{\top} T\ \middle|\ T \in
\mathbb{T}_{++}\right\}.
\]

\subsection{ The dual cones}
\label{section:dualcones}

In general, the dual cone of the intersection of two closed cones is the
Minkowski sum of their respective dual cones when their relative interiors
have a nonempty intersection; see \cite[Corollary~16.4.2]{Rock97}. This
gives the dual cone $\bar{\mathsf{K}}^{*}$ of the Ishi spectrahedral
representation $\mathsf{K} = \mathbb{S}^{n}_{++} \cap \mathbb{V}$ in the
linear subspace $\mathbb{V}$ as $\bar{\mathsf{K}}^{*} = \proj_{\mathbb{V}}
\mathbb{S}^{n}_{+}$. This expression, however, does not reveal the relative
interior $\mathsf{K}^{*}$ of the dual cone $\bar{\mathsf{K}}^{*}$ as a
homogeneous convex cone.

Generally, under any given inner product on $\mathbb{V}$, the adjoint
transformations of any arbitrary transitive group of linear automorphisms
(on $\mathbb{V}$) of a homogeneous convex cone $\mathsf{K}$ in $\mathbb{V}$
form a linear automorphism group of the relative interior $\mathsf{K}^{*}$
of the dual cone of $\mathsf{K}$ (in $\mathbb{V}$) that acts transitively
on it, and hence the relative interior $\mathsf{K}^{*}$ of the dual cone is
the orbit of any arbitrary vector in the relative interior $\mathsf{K}^{*}$
under the action of this transitive group of adjoint transformations; see
\cite[Proposition~I.4.9]{Vin65}. Particularly, under the trace inner
product on $\mathbb{V}$, the identity matrix $I_{n}$ is in the relative
interior $\mathsf{K}^{*}$, and for every $\mathcal{T}_{T} \in
\mathsf{T}_{\mathsf{K}}$ and every $X, Y \in \mathbb{V}$,
\[
\left\langle \mathcal{T}_{T} X \middle| Y \right\rangle_{\mathrm{F}} =
\left\langle T^{\top} X T \middle| Y \right\rangle_{\mathrm{F}} =
\left\langle X \middle| \proj_{\mathbb{V}}\left( T Y T^{\top} \right)
\right\rangle_{\mathrm{F}},
\]
where the final equality follows from $X \in \mathbb{V}$, shows that the
adjoint transformation of each $\mathcal{T}_{T} \in
\mathsf{T}_{\mathsf{K}}$ is $Y \in \mathbb{V} \mapsto
\proj_{\mathbb{V}}\left( T Y T^{\top} \right)$. Therefore,
\[ \mathsf{T}_{\mathsf{K}^{*}} :=
\left\{\mathcal{T}^{*}_{T} : Y \in \mathbb{V} \mapsto
\proj_{\mathbb{V}}\left( T Y T^{\top} \right)\ \middle|\ T \in
\mathbb{T}_{++}\right\}
\]
is a transitive group of linear automorphisms of the relative interior
$\mathsf{K}^{*}$ of the dual cone, which can then be expressed as
\[
\mathsf{K}^{*} = \left\{\mathcal{T}^{*}_{T} I_{n} =
\proj_{\mathbb{V}}\left( T T^{\top} \right)\ \middle|\ T \in
\mathbb{T}_{++}\right\}.
\]

\section{Facial structure of homogeneous convex cones}
\label{section:facialstructure}

A face of a closed convex cone $\bar{\mathsf{K}}$ is a nonempty convex
subset of $\bar{\mathsf{K}}$ that contains both endpoints of every open
line segment in the cone with which it has a nonempty intersection. It is,
therefore, a closed convex cone. The linearity space $\bar{\mathsf{K}} \cap
-\bar{\mathsf{K}}$ of a closed convex cone $\bar{\mathsf{K}}$ is a face
called the trivial face. It is the trivial subspace when $\bar{\mathsf{K}}$
is pointed. A face of $\bar{\mathsf{K}}$ is said to be proper if it is
neither $\bar{\mathsf{K}}$ nor the trivial face. An extreme ray of
$\bar{\mathsf{K}}$ is a face of dimension $1$. A chain of proper faces of
$\bar{\mathsf{K}}$ is a sequence of faces of $\bar{\mathsf{K}}$ in which
every face is a proper face of the preceding face.

Every point of $\bar{\mathsf{K}}$ is contained in the relative interior of
exactly one of its faces, which is the smallest face of $\bar{\mathsf{K}}$
containing the point; see \cite[Theorem~18.2]{Rock97} and its proof. We
shall denote by $\mathsf{F}_{\bar{\mathsf{K}}}(x)$ the smallest face of
$\bar{\mathsf{K}}$ containing the point $x$ of $\bar{\mathsf{K}}$. We begin
with a characterization of faces of linear slices of closed convex cones;
cf.\@ \cite[Exercise~9,~Page~21]{Gru03}, \cite[Proposition~5.8]{Weis24},
and \cite[Proposition~2.13]{ShiWeis21}.

\begin{proposition}\label{proposition:facesofslices}
For a closed convex cone $\bar{\mathsf{K}}$ and a linear subspace
$\mathbb{L}$ in the same vector space, a subset $\mathsf{F}$ of
$\bar{\mathsf{K}} \cap \mathbb{L}$ is a face if and only if it is the
intersection of $\mathbb{L}$ with a face $\mathsf{F}'$ of
$\bar{\mathsf{K}}$, in which case, the face $\mathsf{F}'$ may be taken to
be the smallest face containing $\mathsf{F}$.
\end{proposition}
\begin{proof}
We take an arbitrary point $x$ in the relative interior of a face
$\mathsf{F}$ of $\bar{\mathsf{K}} \cap \mathbb{L}$, which is also in the
relative interior of $\mathsf{F}_{\bar{\mathsf{K}}}(x)$. For any point $y$
in the intersection $\bar{\mathsf{K}} \cap \mathbb{L} \subseteq
\bar{\mathsf{K}}$, it is in the face $\mathsf{F}$ (respectively, the face
$\mathsf{F}_{\bar{\mathsf{K}}}(x)$) if and only if there is some $\epsilon
> 0$ for which $x - \epsilon (y - x)$ is in the intersection
$\bar{\mathsf{K}} \cap \mathbb{L}$ (respectively, the cone
$\bar{\mathsf{K}}$). At the same time, the point $x - \epsilon (y - x)$
with $\epsilon > 0$ lies in the intersection $\bar{\mathsf{K}} \cap
\mathbb{L}$ if and only if it lies in the cone $\bar{\mathsf{K}}$.
Conversely, the definition of faces implies that every face of
$\bar{\mathsf{K}}$ intersects with $\mathbb{L}$ to give a face of
$\bar{\mathsf{K}} \cap \mathbb{L}$.
\end{proof}

\subsection{Faces of homogeneous convex cones}
\label{section:facialstructureprimal}

In \cite[\S2]{Chua06}, the author showed that faces of the positive
semidefinite cone $\mathbb{S}^{n}_{+}$ are of the form
\[
\left\{T^{\top} X T\ \middle|\ X \in \mathbb{S}^{n}_{+} \text{ and } (i
\notin \mathsf{B} \vee j \notin \mathsf{B}) \implies X_{ij} = 0 \ (1 \leq i
\leq j \leq n)\right\},
\]
where $T$ is an upper triangular matrix with positive diagonal and
$\mathsf{B}$ is a subset of $\left\{1, \dotsc, n\right\}$. Moreover, every
positive semidefinite matrix $X \in \mathbb{S}^{n}_{+}$ has a unique
Cholesky-type factorization $X = T^{\top} T$ with
\[
T \in \mathbb{T}^{n}_{+} := \left\{T \in \mathbb{R}^{n \times n}\ \middle|\
\begin{aligned}
&T \text{ is upper triangular, } T_{ii} \geq 0 \ (1 \leq i \leq n),
\\
& (T_{ii} = 0 \implies T_{ij} = 0) \ (1 \leq i < j \leq n)
\end{aligned}
\right\}
\]
and this unique Cholesky-type factorization can be obtained as the limit of
the sequence of Cholesky factorizations of the positive definite sequence
$\left( X + I_{n} / k \right)_{k=1}^\infty$. It was further shown in
\cite{Chua06} that the relative interior of the above face has the same
description as the face, but with the triangular matrix $\hat{T} \in
\mathbb{T}^{n}_{+}$, in the unique Cholesky-type factorization $X =
\hat{T}^{\top} \hat{T}$ of every $X$ in the description, having positive
diagonal entries at positions indexed by $\mathsf{B}$.

We now consider a face $\mathsf{F}$ of the closure $\bar{\mathsf{K}} =
\mathbb{S}^{n}_{+} \cap \mathbb{V}$ of any Ishi spectrahedral
representation $\mathsf{K} = \mathbb{S}^{n}_{++} \cap \mathbb{V}$ of a
homogeneous convex cone of rank $r$ and some matrix $\hat{X}$ in its
relative interior. Since the positive definite matrices in the sequence
$(\hat{X} + I_{n} / k)_{k=1}^\infty$ are in the linear subspace
$\mathbb{V}$, the triangular matrices $\hat{T}(k)$ in the Cholesky
factorizations $\hat{X} + I_{n} / k = \hat{T}(k)^{\top} \hat{T}(k)$ belong
to the linear subspace $\mathbb{T}$; hence, the triangular matrix $\hat{T}
\in \mathbb{T}^{n}_{+}$ with $\hat{X} = \hat{T}^{\top} \hat{T}$ is also in
the linear subspace $\mathbb{T}$ because it is the limit of the sequence
$(\hat{T}(k))_{k=1}^\infty$ in $\mathbb{T}$. The minimal containing face
$\mathsf{F}_{\mathbb{S}^{n}_{+}}(\hat{X} = \hat{T}^{\top} \hat{T})$ is
therefore
\[
\left\{T^{\top} X T\ \middle|\ X \in \mathbb{S}^{n}_{+} \text{ and } (i
\notin \mathsf{B} \vee j \notin \mathsf{B}) \implies X_{ij} = O_{n_{i}
\times n_{j}} \ (1 \leq i \leq j \leq r)\right\},
\]
where, here, the notation $X_{ij}$ refers to the $(i, j)$-th block
according to the block structure of $\mathbb{V}$, $\mathsf{B}$ is the index
set $\{i|\hat{T}_{ii} > 0\} \subseteq \{1, \dotsc, r\}$, and $T \in
\mathbb{T}$ is the upper triangular matrix obtained from $\hat{T}$ by
changing all zeroes on its diagonal to ones. By
Proposition~\ref{proposition:facesofslices}, the face $\mathsf{F}$ of the
closure $\bar{\mathsf{K}} = \mathbb{S}^{n}_{+} \cap \mathbb{V}$ is the
intersection of $\mathbb{V}$ with the above face. Thus, the face
$\mathsf{F}$ is
\[
\mathsf{F}_{T,\mathsf{B}} := \left\{T^{\top} X T\ \middle|\ X \in
\mathbb{S}^{n}_{+} \cap \mathbb{V}^{\mathsf{B}}\right\} = \mathcal{T}_{T}
\left( \mathbb{S}^{n}_{+} \cap \mathbb{V}^{\mathsf{B}} \right),
\]
where
\[
\mathbb{V}^{\mathsf{B}} := \left\{X \in \mathbb{V}\ \middle|\ (i \notin
\mathsf{B} \vee j \notin \mathsf{B}) \implies X_{ij} = O_{n_{i} \times
n_{j}} \ (1 \leq i \leq j \leq r)\right\};
\]
cf.\@ \cite[Proposition~18]{Chua06}.

For example, when $T =  I_{n}$ and $\mathsf{B} = \left\{i_{1}, \dotsc,
i_{p}\right\}$ with $i_{1} < \dotsb < i_{p}$ is a nonempty subset of
$\left\{1, \dotsc, r\right\}$, the resulting face
$\mathsf{F}_{I_{n},\mathsf{B}}$ is the spectrahedral cone
$\mathbb{S}^{n}_{+} \cap \mathbb{V}^{\mathsf{B}}$ whose relative interior
is a homogeneous convex cone of rank $p = \lvert \mathsf{B} \rvert$ because
it has the Ishi spectrahedral representation
$\mathbb{S}^{n_{\mathsf{B}}}_{++} \cap \mathbb{V}_{\mathsf{B}}$, where
$n_{\mathsf{B}} := n_{i_{1}} + \dotsb + n_{i_{p}}$ and
\[
\mathbb{V}_{\mathsf{B}} := \left\{
\begin{pmatrix}
X_{11} & X_{12} & \cdots & X_{1p}
\\
X_{12}^{\top} & X_{22} & \cdots & X_{2p}
\\
\vdots & \vdots & \ddots & \vdots
\\
X_{1p}^{\top} & X_{2p}^{\top} & \cdots & X_{pp}
\end{pmatrix}
\ \middle|\
\begin{aligned}
& X_{kk} \in \mathbb{V}_{kk} = \mathbb{R} I_{n_{i_{k}}} \ (1 \leq k \leq
p),
\\
& X_{kl} \in \mathbb{V}_{i_{k} i_{l}} \ (1 \leq k < l \leq p)
\end{aligned}
\right\},
\]
via the canonical linear isomorphism $\mathbb{V}^{\mathsf{B}}
\xrightarrow{\sim} \mathbb{V}_{\mathsf{B}}$ that removes the zero blocks.
Alternatively, the homogeneity of the relative interior can be observed
from the simply transitive group of linear automorphisms
\[
\mathsf{T}_{\mathsf{F}_{I_{n},\mathsf{B}}} :=
\left\{\mathcal{T}_{T}^{\mathsf{B}} : X \in \mathbb{V}^{\mathsf{B}} \mapsto
T^{\top} X T\ \middle|\ T \in \mathbb{T}_{++}^{\mathsf{B}}\right\},
\]
where $\mathbb{T}_{++}^{\mathsf{B}}$ is the subset of triangular matrices
$T \in \mathbb{T}^{n}_{+}$ with $T^{\top} T$ in the relative interior of
the face $\mathsf{F}_{I_{n},\mathsf{B}} = \mathbb{S}^{n}_{+} \cap
\mathbb{V}^{\mathsf{B}}$.

More generally, the relative interior of each nontrivial face
$\mathsf{F}_{T,\mathsf{B}} = \mathcal{T}_{T} \mathsf{F}_{I_{n},\mathsf{B}}$
is also a homogeneous convex cone of rank $\lvert \mathsf{B} \rvert$
because it is linearly isomorphic to the face
$\mathsf{F}_{I_{n},\mathsf{B}}$ via $\mathcal{T}_{T^{-1}} =
\mathcal{T}_{T}^{-1} \in \mathsf{T}_{\mathsf{K}}$, and also because it has
the transitive group of linear automorphisms
$\mathsf{T}_{\mathsf{F}_{T,\mathsf{B}}} := \mathcal{T}_{T}
\mathsf{T}_{\mathsf{F}_{I_{n},\mathsf{B}}} \mathcal{T}_{T^{-1}}$; cf.\@
\cite[Theorem~3.6]{GouItoLou25}. We shall call each face
$\mathsf{F}_{T,\mathsf{B}}$ with a homogeneous relative interior of rank $p
= \lvert \mathsf{B} \rvert$ a \emph{rank $p$ face} of the closure
$\bar{\mathsf{K}} = \mathbb{S}^{n}_{+} \cap \mathbb{V}$ and every point in
its relative interior a \emph{rank $p$ point} of the closure.

Every rank $\lvert \mathsf{B} \rvert$ face $\mathsf{F}_{T,\mathsf{B}} =
\mathcal{T}_{T}(\mathbb{S}^{n}_{+} \cap \mathbb{V}^{\mathsf{B}})$ of the
closure $\bar{\mathsf{K}}$ with $\lvert \mathsf{B} \rvert < r-1$ is a
proper face of every proper face $\mathsf{F}_{T,\mathsf{B}'} =
\mathcal{T}_{T}(\mathbb{S}^{n}_{+} \cap \mathbb{V}^{\mathsf{B}'})$ with
$\mathsf{B} \subsetneqq \mathsf{B}' \subsetneqq \left\{1, \dotsc,
r\right\}$, and is therefore not a maximal proper face of the closure
$\bar{\mathsf{K}}$. Thus, the maximal faces of the closure
$\bar{\mathsf{K}}$ are precisely the rank $r-1$ faces. Applying this
argument to every proper face, which is also the closure of a homogeneous
convex cone, shows that maximal chains of proper faces of the closure
$\bar{\mathsf{K}}$ have length $r$, hence characterizing the rank
geometrically as the length of maximal chains of proper faces.

Our description of nontrivial faces of the closure $\bar{\mathsf{K}} =
\mathbb{S}^{n}_{+} \cap \mathbb{V}$ shows that its extreme rays are rank
$1$ faces, which are of the form $\mathsf{F}_{T,\left\{i\right\}} =
\mathbb{R}_+ (T^{\top} I^{\left\{i\right\}} T)$, where $I^{\mathsf{B}}$
denotes the orthogonal projection $\proj_{\mathbb{V}^{\mathsf{B}}} I_{n}$
for any index set $\mathsf{B} \subseteq \left\{1, \dotsc, r\right\}$. So,
any matrix $X = \mathcal{T}_{T} I_{n}$ in the relative interior
$\mathbb{S}^{n}_{++} \cap \mathbb{V}$ is in the Minkowski sum of the $r$
extreme rays $\mathsf{F}_{T,\left\{1\right\}}, \dotsc,
\mathsf{F}_{T,\left\{r\right\}}$, and hence the Carath\'eodory number at
$X$ is at most $r$. Applying this argument to every proper face of the
closure $\bar{\mathsf{K}}$ shows that the Carath\'eodory number at every
rank $p$ point is at most $p \leq r$. We have thus recovered the fact that
the rank of a homogeneous convex cone is at least the Carath\'eodory number
of its closure using a facial description of its closure.

\subsection{Faces of dual cones}
\label{section:facialstructuredual}

Every proper face $\mathsf{F}_{T,\mathsf{B}} = \left\{T^{\top} X T\
\middle|\ X \in \mathbb{S}^{n}_{+} \cap \mathbb{V}^{\mathsf{B}}\right\}$ of
the closure $\bar{\mathsf{K}} = \mathbb{S}^{n}_{+} \cap \mathbb{V}$ of any
Ishi spectrahedral representation $\mathsf{K} = \mathbb{S}^{n}_{++} \cap
\mathbb{V}$ of a homogeneous convex cone is exposed by the hyperplane in
$\mathbb{V}$ orthogonal to the matrix $\proj_{\mathbb{V}}\left( T^{-1}
(I_{n} - I^{\mathsf{B}}) T^{-\top} \right)$ in the dual cone
$\bar{\mathsf{K}}^{*} = \proj_{\mathbb{V}} \mathbb{S}^{n}_{+}$.\footnote{In
fact, it is projectionally exposed because it is the image of the closure
$\bar{\mathsf{K}}$ under the projection $X \mapsto T^{\top}
\proj_{\mathbb{V}^{\mathsf{B}}}\left( T^{-\top} X T^{-1} \right) T$; cf.\@
\cite[Theorem~3.6]{GouItoLou25}.} Therefore, every proper face of the dual
cone $\bar{\mathsf{K}}^{*}$, whose relative interior is also a homogeneous
convex cone, is exposed by the hyperplane in its linear span $\mathbb{V}$
that is orthogonal to some nonzero matrix $X$ in the boundary
$\bar{\mathsf{K}} \setminus \mathsf{K}$. In fact, the same proper face of
the dual cone $\bar{\mathsf{K}}^{*}$ is exposed by every matrix in the
relative interior of the minimal containing face
$\mathsf{F}_{\bar{\mathsf{K}}}(X)$. The minimum containing face is the
proper face $\mathsf{F}_{\hat{T},\mathsf{B}}$ of $\bar{\mathsf{K}}$ for
some $\hat{T} \in \mathbb{T}_{++}$ and some nonempty and proper subset
$\mathsf{B}$ of $\left\{1, \dotsc, r\right\}$, and the matrix
$\mathcal{T}_{\hat{T}}I^{\mathsf{B}} = \hat{T}^{\top} I^{\mathsf{B}}
\hat{T}$ in the relative interior of $\mathsf{F}_{\hat{T},\mathsf{B}}$
exposes the proper face of the dual cone $\bar{\mathsf{K}}^{*}$ as
\[
\begin{aligned}
&\left\{\hat{Y} \in \bar{\mathsf{K}}^{*}\ \middle|\ 0 = \left\langle
\hat{Y} \middle| \mathcal{T}_{\hat{T}} I^{\mathsf{B}}
\right\rangle_{\mathrm{F}} = \left\langle \mathcal{T}^{*}_{\hat{T}} \hat{Y}
\middle| I^{\mathsf{B}} \right\rangle_{\mathrm{F}}\right\}
\\
&= \mathcal{T}^{*}_{\hat{T}^{-1}} \left\{Y \in \mathcal{T}^{*}_{\hat{T}}
\bar{\mathsf{K}}^{*} = \bar{\mathsf{K}}^{*} = \proj_{\mathbb{V}}
\mathbb{S}^{n}_{+}\ \middle|\ 0 = \left\langle Y \middle| I^{\mathsf{B}}
\right\rangle_{\mathrm{F}}\right\}
\\
&= \mathcal{T}^{*}_{T} \proj_{\mathbb{V}} \left\{Z \in \mathbb{S}^{n}_{+}\
\middle|\ 0 = \left\langle \proj_{\mathbb{V}} Z \middle| I^{\mathsf{B}}
\right\rangle_{\mathrm{F}} = \left\langle Z \middle| I^{\mathsf{B}}
\right\rangle_{\mathrm{F}}\right\} & (\text{with } T = \hat{T}^{-1})
\\
&= \mathcal{T}^{*}_{T} \proj_{\mathbb{V}^{\mathsf{N}}} \mathbb{S}^{n}_{+} &
(\text{with } \mathsf{N} = \left\{1, \dotsc, r\right\} \setminus
\mathsf{B}).
\end{aligned}
\]
This gives nontrivial faces of the dual cone $\bar{\mathsf{K}}^{*}$ as
\[
\mathsf{F}^{*}_{T,\mathsf{N}} := \left\{\proj_{\mathbb{V}}\left( T Y
T^{\top} \right)\ \middle|\ Y \in \proj_{\mathbb{V}^{\mathsf{N}}}
\mathbb{S}^{n}_{+}\right\} = \mathcal{T}^{*}_{T} \left(
\proj_{\mathbb{V}^{\mathsf{N}}} \mathbb{S}^{n}_{+} \right),
\]
where $T \in \mathbb{T}_{++}$ and $\mathsf{N}$ is a nonempty subset of
$\left\{1, \dotsc, r\right\}$.

In particular, the face $\mathsf{F}^{*}_{I_{n},\mathsf{N}}$ is
$\proj_{\mathbb{V}^{\mathsf{N}}}$ $\mathbb{S}^{n}_{+}$, which is the dual
cone of the face $\mathsf{F}_{I_{n},\mathsf{N}} = \mathbb{S}^{n}_{+} \cap
\mathbb{V}^{\mathsf{N}}$ in the linear subspace $\mathbb{V}^{\mathsf{N}}$.
Hence, its homogeneous relative interior has the simply transitive group of
linear automorphisms
\[
\mathsf{T}_{\mathsf{F}^{*}_{I_{n},\mathsf{N}}} :=
\left\{{\mathcal{T}_{T}^{\mathsf{N}}}^{*} : Y \in \mathbb{V}^{\mathsf{N}}
\mapsto \proj_{\mathbb{V}}\left( T Y T^{\top} \right)\ \middle|\ T \in
\mathbb{T}_{++}^{\mathsf{N}}\right\}
\]
and can be described as the set
\[
\left\{{\mathcal{T}_{T}^{\mathsf{N}}}^{*} I^{\mathsf{N}} =
\proj_{\mathbb{V}}\left( T T^{\top} \right)\ \middle|\ T \in
\mathbb{T}_{++}^{\mathsf{N}}\right\}.
\]
Taking the closure then gives
\[
\mathsf{F}^{*}_{I_{n},\mathsf{N}} = \left\{\proj_{\mathbb{V}}\left( T
T^{\top} \right)\ \middle|\ T \in \mathbb{T}_{+}^{\mathsf{N}}\right\}
\]
where $\mathbb{T}_{+}^{\mathsf{N}}$ denotes the closure of
$\mathbb{T}_{++}^{\mathsf{N}}$ in $\mathbb{R}^{n \times n}$.

When $\mathsf{N} = \left\{1, \dotsc, i\right\}$ for some $i \in \left\{2,
\dotsc, r\right\}$, any matrix $Y$ in the face
$\mathsf{F}^{*}_{I_{n},\mathsf{N}}$ but not the proper face
$\mathsf{F}^{*}_{I_{n},\mathsf{N}'}$, where $\mathsf{N}' = \left\{1,
\dotsc, i-1\right\}$, can be expressed as $Y = \proj_{\mathbb{V}}\left( T
T^{\top} \right)$ for some $T \in \mathbb{T}_{+}^{\mathsf{N}}$ whose
diagonal block $T_{ii}$ must be nonzero, for otherwise
$\proj_{\mathbb{V}}\left( T T^{\top} \right)$ would be in
$\mathsf{F}^{*}_{I_{n},\mathsf{N}'} = \proj_{\mathbb{V}^{\mathsf{N}'}}
\mathbb{S}^{n}_{+}$. We decompose the triangular matrix $T = U + \hat{T}$,
where $U$ is the projection of $T$ onto the linear span of
$\mathbb{T}_{+}^{\mathsf{N}'}$, so that $Y = \proj_{\mathbb{V}}\left( T
T^{\top} \right)$ decomposes into $\proj_{\mathbb{V}}\left( U U^{\top}
\right) + \proj_{\mathbb{V}}\left( \hat{T} \hat{T}^{\top} \right)$. So, any
face of $\mathsf{F}^{*}_{I_{n},\mathsf{N}}$ that contains $Y$ must contain
the average of the two matrices $\proj_{\mathbb{V}}\left( U U^{\top}
\right)$ and $\proj_{\mathbb{V}}\left( \hat{T} \hat{T}^{\top} \right)$ in
$\mathsf{F}^{*}_{I_{n},\mathsf{N}}$, and hence the matrix
$\proj_{\mathbb{V}}\left( \hat{T} \hat{T}^{\top} \right)$. If this face
further contains the proper face $\mathsf{F}^{*}_{I_{n},\mathsf{N}'}$, then
it must contain the Minkowski sum of $\mathsf{F}^{*}_{I_{n},\mathsf{N}'}$
and $\proj_{\mathbb{V}}\left( \hat{T} \hat{T}^{\top} \right)$, which in
turn contains the matrix $\proj_{\mathbb{V}}\left( (I^{\mathsf{N}'} +
\hat{T}) (I^{\mathsf{N}'} + \hat{T})^{\top} \right) =
\proj_{\mathbb{V}}\left( (I^{\mathsf{N}'}) (I^{\mathsf{N}'})^{\top} \right)
+ \proj_{\mathbb{V}}\left( \hat{T} \hat{T}^{\top} \right)$ in the relative
interior of $\mathsf{F}^{*}_{I_{n},\mathsf{N}}$, and must thus be
$\mathsf{F}^{*}_{I_{n},\mathsf{N}}$ itself. This means that no proper face
of $\mathsf{F}^{*}_{I_{n},\mathsf{N}}$ can properly contain
$\mathsf{F}^{*}_{I_{n},\mathsf{N}'}$; i.e.,
$\mathsf{F}^{*}_{I_{n},\mathsf{N}'}$ is a maximal proper face of
$\mathsf{F}^{*}_{I_{n},\mathsf{N}}$. Therefore, the chain of proper faces
$\mathsf{F}^{*}_{I_{n},\left\{1, \dotsc, r\right\}},
\mathsf{F}^{*}_{I_{n},\left\{1, \dotsc, r-1\right\}}, \dotsc,
\mathsf{F}^{*}_{I_{n},\left\{1\right\}}$ is maximal and its length $r$ is
then the rank of the homogeneous relative interior $\mathsf{K}^{*}$. This
is a geometrical evidence that the rank of the homogeneous relative
interior of the dual cone agrees with that of the homogeneous convex cone.

More generally, the nontrivial face $\mathsf{F}^{*}_{T,\mathsf{N}} =
\mathcal{T}^{*}_{T} \mathsf{F}^{*}_{I_{n},\mathsf{N}}$ has a homogeneous
relative interior with the transitive group of linear automorphisms
$\mathcal{T}^{*}_{T} \mathsf{T}_{\mathsf{F}^{*}_{I_{n},\mathsf{N}}}
\mathcal{T}^{*}_{T^{-1}}$; cf.\@ \cite[Theorem~3.11]{GouItoLou25}. So, the
face $\mathsf{F}^{*}_{T,\mathsf{N}}$ is a rank $\lvert \mathsf{N} \rvert$
face of the dual cone and, in particular, the extreme rays of the dual cone
$\bar{\mathsf{K}}^{*}$ are of the form $\mathsf{F}^{*}_{T,\left\{i\right\}}
= \mathbb{R}_+ \proj_{\mathbb{V}}\left( T I^{\left\{i\right\}} T^{\top}
\right)$.

\section{Carath\'eodory numbers of homogeneous convex cones}
\label{section:Caratheodorynumbers}

Since every extreme ray of the closure $\bar{\mathsf{K}} =
\mathbb{S}^{n}_{+} \cap \mathbb{V}$ of any Ishi spectrahedral
representation $\mathsf{K} = \mathbb{S}^{n}_{++} \cap \mathbb{V}$ of a
homogeneous convex cone is of the form $\mathsf{F}_{T,\left\{i\right\}} =
\mathbb{R}_+ (T^{\top} I^{\left\{i\right\}} T) = \mathcal{T}_{T}
\mathsf{F}_{I_{n},\left\{i\right\}}$, the collection of all extreme rays is
invariant under the action of the transitive group of linear automorphisms
$\mathsf{T}_{\mathsf{K}}$ of $\mathsf{K}$. So, the function
$\kappa_{\bar{\mathsf{K}}}(\cdot)$ mapping each matrix $X$ in
$\bar{\mathsf{K}}$ to its Carath\'eodory number
$\kappa_{\bar{\mathsf{K}}}(X)$ is invariant under the action of
$\mathsf{T}_{\mathsf{K}}$. In particular, all matrices in the relative
interiors of all faces $\mathsf{F}_{T,\mathsf{B}}$ that share the same
index set $\mathsf{B}$ share the same Carath\'eodory number because they
are in the same $\mathsf{T}_{\mathsf{K}}$-orbit as $I^{\mathsf{B}}$. We
denote this number by $\kappa_{\bar{\mathsf{K}},\mathsf{B}}$ and observe
that
\[
\kappa_{\mathsf{F}_{T,\mathsf{B}}} = \max
\left\{\kappa_{\bar{\mathsf{K}},\mathsf{B}'}\ \middle|\ \mathsf{B}'
\subseteq \mathsf{B}\right\}
\]
is invariant under the action of $\mathsf{T}_{\mathsf{K}}$.

Similarly, the extreme rays of the dual cone $\bar{\mathsf{K}}^{*} =
\proj_{\mathbb{V}} \mathbb{S}^{n}_{+}$ are of the form
$\mathsf{F}^{*}_{T,\left\{i\right\}} = \mathbb{R}_+
\proj_{\mathbb{V}}\left( T I^{\left\{i\right\}} T^{\top} \right) =
\mathcal{T}^{*}_{T} \mathsf{F}^{*}_{I_{n},\left\{i\right\}}$ and the
collection of all extreme rays of the dual cone $\bar{\mathsf{K}}^{*}$ is
invariant under the action of the transitive group of linear automorphisms
$\mathsf{T}_{\mathsf{K}^{*}}$ of its relative interior $\mathsf{K}^{*}$.
Hence, as before, all matrices in the relative interiors of all faces
$\mathsf{F}^{*}_{T,\mathsf{N}}$ of the dual cone $\bar{\mathsf{K}}^{*}$
that share the same index set $\mathsf{N}$ share the same Carath\'eodory
number, which we denote by $\kappa_{\bar{\mathsf{K}}^{*},\mathsf{N}}$, and
\[
\kappa_{\mathsf{F}^{*}_{T,\mathsf{N}}} = \max
\left\{\kappa_{\bar{\mathsf{K}}^{*},\mathsf{N}'}\ \middle|\ \mathsf{N}'
\subseteq \mathsf{N}\right\}
\]
is invariant under the action of $\mathsf{T}_{\mathsf{K}^{*}}$.

In \cite[Lemma~4.1]{GulTun98} the authors showed that if a homogeneous
convex cone is selfdual, then its rank matches the Carath\'eodory number of
its closure. On the other hand, when the homogeneous convex cone is not
selfdual, the two numbers may still agree, as is evident from the smallest
example of a non-selfdual homogeneous convex cone $\mathbb{S}^{3}_{++} \cap
\mathbb{V}$ of rank $3$, where
\[
\mathbb{V} = \left\{
\begin{pmatrix}
x_{1} & 0 & x_{4}
\\
0 & x_{2} & x_{5}
\\
x_{4} & x_{5} & x_{3}
\end{pmatrix}
\ \middle|\ x \in \mathbb{R}^5\right\}.
\]
The Carath\'eodory number of its dual cone $\proj_{\mathbb{V}}
\mathbb{S}^{3}_{+}$, however, is less than the rank of its homogeneous
relative interior. For example, the matrix $I_{3} = \proj_{\mathbb{V}}
I_{3}$ in the homogeneous relative interior is in the Minkowski sum of the
two extreme rays $\mathsf{F}^{*}_{T,\left\{3\right\}}$ and
$\mathsf{F}^{*}_{U,\left\{3\right\}}$ of the dual cone, where
\[
T =
\begin{pmatrix}
1 & 0 & 1
\\
0 & 1 & 1
\\
0 & 0 & 1
\end{pmatrix}
\quad \text{and} \quad U =
\begin{pmatrix}
1 & 0 & -1
\\
0 & 1 & -1
\\
0 & 0 & 1
\end{pmatrix}.
\]

We shall show that this example of a homogeneous convex cone with its rank
greater than the Carath\'eodory number of its closure generalizes to
homogeneous convex cones of higher ranks. Before that, since this
exposition, and those that follow in this section, rely heavily on the
products between the block components $\mathbb{V}_{ij}$ of the block matrix
form of the linear subspace $\mathbb{V}$, we first establish a list of
useful identities relating these products.

\subsection{The left and right multiplication operators}
\label{section:leftrightmultiplications}

For each matrix $A_{ij} \in \mathbb{V}_{ij}$ with $1 \leq i \leq j \leq r$,
\begin{enumerate}[label=(\roman*)]
\item the property \ref{axiom:triangularproductblock} is equivalent to the
left multiplication
\[
\mathcal{L}_{A_{ij}} : X_{jl} \in \mathbb{V}_{jl} \mapsto A_{ij} X_{jl}
\]
mapping into the linear subspace $\mathbb{V}_{il}$ for every $l \in
\left\{j, \dotsc, r\right\}$ and, equivalently, the right multiplication
\[
\mathcal{R}_{A_{ij}} : X_{li} \in \mathbb{V}_{li} \mapsto X_{li} A_{ij}
\]
mapping into the linear subspace $\mathbb{V}_{lj}$ for every $l \in
\left\{1, \dotsc, i\right\}$, under which

\item the property \ref{axiom:mixedtriangularproductblock} is then
equivalent to the adjoint of $\mathcal{L}_{A_{ij}}$ on the domain
$\mathbb{V}_{jl}$ taking the form
\[
\mathcal{L}^{*}_{A_{ij}} : X_{il} \in \mathbb{V}_{il} \mapsto
\proj_{\mathbb{V}_{jl}}\left( A_{ij}^{\top} X_{il} \right) = A_{ij}^{\top}
X_{il}
\]
for every $l \in \left\{j+1, \dotsc, r\right\}$ and

\item the property \ref{axiom:selfproductblock}, via polarization, is then
equivalent to the adjoint of $\mathcal{L}_{A_{ij}}$ on the domain
$\mathbb{V}_{jj}$ taking the form
\[
\mathcal{L}^{*}_{A_{ij}} : X_{ij} \in \mathbb{V}_{ij} \mapsto
\proj_{\mathbb{V}_{jj}}\left( A_{ij}^{\top} X_{ij} \right) = \frac{1}{2}
\left( A_{ij}^{\top} X_{ij} + X_{ij}^{\top} A_{ij} \right).
\]

\end{enumerate}
So, all three properties together imply, for every $l \in \left\{j, \dotsc,
r\right\}$,
\[
\begin{split}
\mathcal{L}^{*}_{A_{ij}} \mathcal{L}_{A_{ij}} : X_{jl} \in \mathbb{V}_{jl}
&\mapsto \mathcal{L}^{*}_{A_{ij}} (A_{ij} X_{jl}) = A_{ij}^{\top} A_{ij}
X_{jl}
\\
&= \frac{1}{n_{j}} \trace\left( A_{ij}^{\top} A_{ij} \right) I_{n_{j}}
X_{jl}
= \frac{1}{n_{j}} \left\lVert A_{ij}\right\rVert_{F}^{2} X_{jl}
\end{split}
\]
which means that $\mathcal{L}_{A_{ij}}$ is an injection while its adjoint
is a surjection whenever $A_{ij}$ is nonzero.

We note that for simplicity of notation, we have introduced a slight
ambiguity in notation where each of $\mathcal{L}_{A_{ij}}$ and
$\mathcal{R}_{A_{ij}}$ is used to denote similar linear transformations
with different domains. This shall not be an issue because the domain will
be clear from the context in which they are used.

 We further observe that when studying the geometry of the
dual cone via Ishi spectrahedral representations, the orthogonal projector
that naturally appears can be subsumed in the adjoint transformation
$\mathcal{R}^{*}_{A_{ij}}$ of the right multiplication when dealing with
individual blocks. The adjoint transformations $\mathcal{L}^{*}_{A_{ij}}$
and $\mathcal{R}^{*}_{A_{ij}}$ then allow for a primal-dual symmetric
presentation of the statements and arguments on the geometry of homogeneous
convex cones and their dual cones. For aesthetic reasons\footnote{``Beauty
is the first test: there is no permanent place in the world for ugly
mathematics.'' (G.~H. Hardy, \textit{A Mathematician's Apology}, 1940)}, we
shall use the left and right multiplications and their adjoints, where
appropriate, to maintain the symmetry between mathematical statements, and
arguments, for the primal and the dual cones respectively.

The associativity of matrix multiplication means that for matrices
$(A_{ij}, B_{jk}) \in \mathbb{V}_{ij} \times \mathbb{V}_{jk}$ with $1 \leq
i \leq j \leq r$,
\[
\mathcal{L}_{A_{ij}} \mathcal{L}_{B_{jk}} X_{kl} = A_{ij} (B_{jk} X_{kl}) =
(A_{ij} B_{jk}) X_{kl} = (\mathcal{L}_{A_{ij}} B_{jk}) X_{kl} =
\mathcal{L}_{\mathcal{L}_{A_{ij}} B_{jk}} X_{kl}
\]
whenever $X_{kl} \in \mathbb{V}_{kl}$ with $l \in \left\{k, \dotsc,
r\right\}$, so that $\mathcal{L}_{A_{ij}} \mathcal{L}_{B_{jk}} =
\mathcal{L}_{\mathcal{L}_{A_{ij}} B_{jk}}$. Similarly, the associativity of
matrix multiplication also gives the following identities for matrices
$(A_{ij}, B_{jk}, C_{kl}, D_{ik}) \in \mathbb{V}_{ij} \times
\mathbb{V}_{jk} \times \mathbb{V}_{kl} \times \mathbb{V}_{ik}$ with $1 \leq
i \leq j \leq k \leq l \leq r$.
\begin{enumerate}[label={[\Alph*]},
ref={[\Alph*]}]
\item\label{identity:leftleft}%
$\mathcal{L}_{A_{ij}} \mathcal{L}_{B_{jk}} =
\mathcal{L}_{\mathcal{L}_{A_{ij}} B_{jk}}$ and, by taking the adjoint,
$\mathcal{L}^{*}_{B_{jk}} \mathcal{L}^{*}_{A_{ij}} =
\mathcal{L}^{*}_{\mathcal{L}_{A_{ij}} B_{jk}}$.

\item\label{identity:leftright}%
$\mathcal{L}_{A_{ij}} \mathcal{R}_{C_{kl}} = \mathcal{R}_{C_{kl}}
\mathcal{L}_{A_{ij}}$ and, by taking the adjoint, $\mathcal{R}^{*}_{C_{kl}}
\mathcal{L}^{*}_{A_{ij}} = \mathcal{L}^{*}_{A_{ij}}
\mathcal{R}^{*}_{C_{kl}}$.

\item\label{identity:rightright}%
$\mathcal{R}_{C_{kl}} \mathcal{R}_{B_{jk}} =
\mathcal{R}_{\mathcal{R}_{C_{kl}} B_{jk}}$ and, by taking the adjoint,
$\mathcal{R}^{*}_{B_{jk}} \mathcal{R}^{*}_{C_{kl}} =
\mathcal{R}^{*}_{\mathcal{R}_{C_{kl}} B_{jk}}$.

\item\label{identity:leftadjointleft}%
$\mathcal{L}^{*}_{A_{ij}} \mathcal{L}_{D_{ik}} =
\mathcal{L}_{\mathcal{L}^{*}_{A_{ij}} D_{ik}}$ if $j < k$ and, by taking
the adjoint, $\mathcal{L}^{*}_{D_{ik}} \mathcal{L}_{A_{ij}} =
\mathcal{L}^{*}_{\mathcal{L}^{*}_{A_{ij}} D_{ik}}$.

\item\label{identity:leftadjointright}%
$\mathcal{L}^{*}_{A_{ij}} \mathcal{R}_{C_{kl}} = \mathcal{R}_{C_{kl}}
\mathcal{L}^{*}_{A_{ij}}$ if $j < k$ and, by taking the adjoint,
$\mathcal{R}^{*}_{C_{kl}} \mathcal{L}_{A_{ij}} = \mathcal{L}_{A_{ij}}
\mathcal{R}^{*}_{C_{kl}}$.

\end{enumerate}
From the adjoint identity in \ref{identity:leftadjointright}, we can
further deduce that when $i < j$,
\[
\mathcal{R}^{*}_{B_{jk}} \mathcal{R}_{D_{ik}} X_{li}
= \mathcal{R}^{*}_{B_{jk}} \mathcal{L}_{X_{li}} D_{ik}
= \mathcal{L}_{X_{li}} \mathcal{R}^{*}_{B_{jk}} D_{ik}
= \mathcal{R}_{\mathcal{R}^{*}_{B_{jk}} D_{ik}} X_{li}
\]
for all $X_{li} \in \mathbb{V}_{li}$ with $l \in \left\{1, \dotsc,
i\right\}$, which gives the identities
\begin{enumerate}[label={[\Alph*]},
ref={[\Alph*]}]
\addtocounter{enumi}{5}%
\item\label{identity:rightadjointright}%
$\mathcal{R}^{*}_{B_{jk}} \mathcal{R}_{D_{ik}} =
\mathcal{R}_{\mathcal{R}^{*}_{B_{jk}} D_{ik}}$ if $i < j$ and, by taking
the adjoint, $\mathcal{R}^{*}_{D_{ik}} \mathcal{R}_{B_{jk}} =
\mathcal{R}^{*}_{\mathcal{R}^{*}_{B_{jk}} D_{ik}}$.

\end{enumerate}

Recall that the left multiplication $\mathcal{L}_{A_{ij}}$ and its adjoint
satisfy the identities
\begin{enumerate}[label={[\Alph*]},
ref={[\Alph*]}]
\addtocounter{enumi}{6}%
\item\label{identity:leftadjointself}%
$\mathcal{L}^{*}_{A_{ij}} \mathcal{L}_{A_{ij}} = \frac{1}{n_{j}}
\left\lVert A_{ij}\right\rVert_{F}^{2} \mathcal{I}_{\mathbb{V}_{jk}}$ and,
by polarization, $\mathcal{L}^{*}_{A_{ij}} \mathcal{L}_{B_{ij}} +
\mathcal{L}^{*}_{B_{ij}} \mathcal{L}_{A_{ij}} = \frac{2}{n_{j}}
\left\langle A_{ij} \middle| B_{ij} \right\rangle_{\mathrm{F}}
\mathcal{I}_{\mathbb{V}_{jk}}$,

\end{enumerate}
where $\mathcal{I}_{\mathbb{V}_{jk}}$ denotes the identity map on
$\mathbb{V}_{jk}$. Since the adjoint identity in \ref{identity:leftright}
implies that every $(X_{ij}, Y_{ij}) \in \mathbb{V}_{ij} \times
\mathbb{V}_{ij}$ satisfies
\[
\begin{split}
\left\langle Y_{ij} \middle| \mathcal{R}^{*}_{B_{jk}} \mathcal{R}_{B_{jk}}
X_{ij} \right\rangle_{\mathrm{F}}
&= \left\langle \mathcal{L}_{Y_{ij}} I_{n_{j}} \middle|
\mathcal{R}^{*}_{B_{jk}} \mathcal{L}_{X_{ij}} B_{jk}
\right\rangle_{\mathrm{F}}
\\
&= \left\langle I_{n_{j}} \middle| \mathcal{L}^{*}_{Y_{ij}}
\mathcal{R}^{*}_{B_{jk}} \mathcal{L}_{X_{ij}} B_{jk}
\right\rangle_{\mathrm{F}}
= \trace\left( \mathcal{R}^{*}_{B_{jk}} \mathcal{L}^{*}_{Y_{ij}}
\mathcal{L}_{X_{ij}} B_{jk} \right)
\end{split}
\]
and, similarly,
\[
\begin{split}
\left\langle Y_{ij} \middle| \mathcal{R}^{*}_{B_{jk}} \mathcal{R}_{B_{jk}}
X_{ij} \right\rangle_{\mathrm{F}}
&= \left\langle \mathcal{R}^{*}_{B_{jk}} \mathcal{R}_{B_{jk}} Y_{ij}
\middle| X_{ij} \right\rangle_{\mathrm{F}}
= \trace\left( \mathcal{R}^{*}_{B_{jk}} \mathcal{L}^{*}_{X_{ij}}
\mathcal{L}_{Y_{ij}} B_{jk} \right),
\end{split}
\]
the polarized identity in \ref{identity:leftadjointself} gives
\[
\begin{split}
\left\langle Y_{ij} \middle| \mathcal{R}^{*}_{B_{jk}} \mathcal{R}_{B_{jk}}
X_{ij} \right\rangle_{\mathrm{F}}
&= \frac{1}{2} \trace\left( \mathcal{R}^{*}_{B_{jk}}
\mathcal{L}^{*}_{Y_{ij}} \mathcal{L}_{X_{ij}} B_{jk} +
\mathcal{R}^{*}_{B_{jk}} \mathcal{L}^{*}_{X_{ij}} \mathcal{L}_{Y_{ij}}
B_{jk} \right)
\\
&= \frac{1}{2} \trace\left( \mathcal{R}^{*}_{B_{jk}} \left(
\mathcal{L}^{*}_{Y_{ij}} \mathcal{L}_{X_{ij}} + \mathcal{L}^{*}_{X_{ij}}
\mathcal{L}_{Y_{ij}} \right) B_{jk} \right)
\\
&= \frac{1}{n_{j}} \left\langle Y_{ij} \middle| X_{ij}
\right\rangle_{\mathrm{F}} \trace\left( \mathcal{R}^{*}_{B_{jk}}
\mathcal{I}_{\mathbb{V}_{jk}} B_{jk} \right)
\\
&= \frac{1}{n_{j}} \left\langle Y_{ij} \middle| X_{ij}
\right\rangle_{\mathrm{F}} \left\langle I_{n_{j}} \middle|
\mathcal{R}^{*}_{B_{jk}} B_{jk} \right\rangle_{\mathrm{F}}
\\
&= \frac{1}{n_{j}} \left\langle Y_{ij} \middle| X_{ij}
\right\rangle_{\mathrm{F}} \left\langle \mathcal{R}_{B_{jk}} I_{n_{j}}
\middle| B_{jk} \right\rangle_{\mathrm{F}}
= \frac{1}{n_{j}} \left\langle Y_{ij} \middle| X_{ij}
\right\rangle_{\mathrm{F}} \left\lVert B_{jk}\right\rVert_{F}^{2},
\end{split}
\]
which shows that $\mathcal{R}^{*}_{B_{jk}} \mathcal{R}_{B_{jk}} =
\frac{1}{n_{j}} \left\lVert B_{jk}\right\rVert_{F}^{2}
\mathcal{I}_{\mathbb{V}_{ij}}$. This gives us the identities
\begin{enumerate}[label={[\Alph*]},
ref={[\Alph*]}]
\addtocounter{enumi}{7}%
\item\label{identity:rightadjointself}%
$\mathcal{R}^{*}_{B_{jk}} \mathcal{R}_{B_{jk}} = \frac{1}{n_{j}}
\left\lVert B_{jk}\right\rVert_{F}^{2} \mathcal{I}_{\mathbb{V}_{ij}}$ and,
by polarization, $\mathcal{R}^{*}_{A_{jk}} \mathcal{R}_{B_{jk}} +
\mathcal{R}^{*}_{B_{jk}} \mathcal{R}_{A_{jk}} = \frac{2}{n_{j}}
\left\langle A_{jk} \middle| B_{jk} \right\rangle_{\mathrm{F}}
\mathcal{I}_{\mathbb{V}_{ij}}$.

\end{enumerate}

\subsection{A necessary condition for matching rank and Carath\'eodory
number}
\label{section:Caratheodorynumbersnecessary}

We now show that the above example of a homogeneous convex cone of rank $3$
with mismatched rank and Carath\'eodory number may be generalized to
homogeneous convex cones of higher ranks.

We consider, for any Ishi spectrahedral representation $\mathsf{K} =
\mathbb{S}^{n}_{++} \cap \mathbb{V}$ of a homogeneous convex cone of rank
$r$, three indices $1 \leq i < j < k \leq r$ and two extreme rays
$\mathsf{F}_{T,\left\{i\right\}} = \mathbb{R}_+ (T^{\top}
I^{\left\{i\right\}} T)$ and $\mathsf{F}_{U,\left\{i\right\}} =
\mathbb{R}_+ (U^{\top} I^{\left\{i\right\}} U)$ of its closure
$\bar{\mathsf{K}} = \mathbb{S}^{n}_{+} \cap \mathbb{V}$, where
\[
T_{ii} = U_{ii} = I_{n_{i}}, \quad T_{ij} = -U_{ij} = A_{ij}, \quad
\text{and} \quad T_{ik} = -U_{ik} = B_{ik}
\]
for some nonzero matrices $A_{ij} \in \mathbb{V}_{ij}$ and $B_{ik} \in
\mathbb{V}_{ik}$. The matrices $T^{\top} I^{\left\{i\right\}} T$ and
$U^{\top} I^{\left\{i\right\}} U$ have
\[
\begin{aligned}
(T^{\top} I^{\left\{i\right\}} T)_{ii} &= (U^{\top} I^{\left\{i\right\}}
U)_{ii} = I_{n_{i}}, \quad
\\
(T^{\top} I^{\left\{i\right\}} T)_{ij} &= -(U^{\top} I^{\left\{i\right\}}
U)_{ij} = A_{ij}, \quad
\\
(T^{\top} I^{\left\{i\right\}} T)_{ik} &= -(U^{\top} I^{\left\{i\right\}}
U)_{ik} = B_{ik},
\\
(T^{\top} I^{\left\{i\right\}} T)_{jj} &= (U^{\top} I^{\left\{i\right\}}
U)_{jj} = A_{ij}^{\top} A_{ij} = \mathcal{L}^{*}_{A_{ij}}
\mathcal{L}_{A_{ij}} I_{n_{j}} \overset{\ref{identity:leftadjointself}}{=}
\frac{1}{n_{j}} \left\lVert A_{ij}\right\rVert_{F}^{2} I_{n_{j}},
\\
(T^{\top} I^{\left\{i\right\}} T)_{kk} &= (U^{\top} I^{\left\{i\right\}}
U)_{kk} = B_{ik}^{\top} B_{ik} = \mathcal{L}^{*}_{B_{ik}}
\mathcal{L}_{B_{ik}} I_{n_{k}} \overset{\ref{identity:leftadjointself}}{=}
\frac{1}{n_{k}} \left\lVert B_{ik}\right\rVert_{F}^{2} I_{n_{k}},
\end{aligned}
\]
and zero blocks everywhere else. Therefore, their average is the matrix
$I^{\left\{i\right\}} + \hat{T}^{\top} I^{\left\{j,k\right\}} \hat{T}$,
where the matrix $\hat{T}$ has
\[
\begin{aligned}
\hat{T}_{jj} &= \sqrt{(T^{\top} I^{\left\{i\right\}} T)_{jj}} =
\frac{1}{\sqrt{n_{j}}} \left\lVert A_{ij}\right\rVert_{F} I_{n_{j}},
\\
\hat{T}_{jk} &= \hat{T}_{jj}^{-1} (T^{\top} I^{\left\{i\right\}}
T)_{ij}^{\top} (T^{\top} I^{\left\{i\right\}} T)_{ik} = \sqrt{n_{j}}
\left\lVert A_{ij}\right\rVert_{F}^{-1} A_{ij}^{\top} B_{ik} = \sqrt{n_{j}}
\left\lVert A_{ij}\right\rVert_{F}^{-1} \mathcal{L}^{*}_{A_{ij}} B_{ik},
\\
\hat{T}_{kk} &= \sqrt{(T^{\top} I^{\left\{i\right\}} T)_{kk} -
\hat{T}_{jk}^{\top} \hat{T}_{jk}} = \sqrt{\frac{1}{n_{k}} \left\lVert
B_{ik}\right\rVert_{F}^{2} - \frac{n_{j}}{n_{k}} \left\lVert
A_{ij}\right\rVert_{F}^{-2} \left\lVert \mathcal{L}^{*}_{A_{ij}}
B_{ik}\right\rVert_{F}^{2}} I_{n_{k}},
\end{aligned}
\]
and zero blocks everywhere else. This average is a rank $3$ point if
\[
\left\lVert A_{ij}\right\rVert_{F}^{2} \left\lVert
B_{ik}\right\rVert_{F}^{2} > n_{j} \left\lVert \mathcal{L}^{*}_{A_{ij}}
B_{ik}\right\rVert_{F}^{2},
\]
in which case the matrix
\[
\sum_{l \in \left\{1, \dots, r\right\} \setminus \left\{j,k\right\}} I^{l}
+ \hat{T}^{\top} I^{\left\{j,k\right\}} \hat{T} = \left( \sum_{l \in
\left\{1, \dots, r\right\} \setminus \left\{j,k\right\}} I^{l} + \hat{T}
\right)^{\top} \left( \sum_{l \in \left\{1, \dots, r\right\} \setminus
\left\{j,k\right\}} I^{l} + \hat{T} \right)
\]
is in the relative interior $\mathbb{S}^{n}_{++} \cap \mathbb{V}$ and also
in the Minkowski sum of the $r-1$ extreme rays
$\mathsf{F}_{T,\left\{i\right\}}$, $\mathsf{F}_{U,\left\{i\right\}}$, and
$\mathsf{F}_{I_{n},\left\{l\right\}}$ for $l \in \left\{1, \dots, r\right\}
\setminus \left\{i,j,k\right\}$, which shows that the Carath\'eodory number
$\kappa_{\bar{\mathsf{K}},\left\{1, \dotsc, r\right\}}$ at points in the
relative interior, and hence that of the closure $\bar{\mathsf{K}}$, is at
most $r-1$.

Dually, we consider, for the dual cone $\bar{\mathsf{K}}^{*} =
\proj_{\mathbb{V}} \mathbb{S}^{n}_{+}$ whose relative interior is a
homogeneous convex cone of rank $r$, three indices $1 \leq k < j < i \leq
r$ and two extreme rays $\mathsf{F}^{*}_{T,\left\{i\right\}} = \mathbb{R}_+
\proj_{\mathbb{V}}\left( T I^{\left\{i\right\}} T^{\top} \right)$ and
$\mathsf{F}^{*}_{U,\left\{i\right\}} = \mathbb{R}_+
\proj_{\mathbb{V}}\left( U I^{\left\{i\right\}} U^{\top} \right)$ of the
dual cone $\bar{\mathsf{K}}^{*}$, where
\[
T_{ii} = U_{ii} = I_{n_{i}}, \quad T_{ji} = -U_{ji} = A_{ji}, \quad
\text{and} \quad T_{ki} = -U_{ki} = B_{ki}
\]
for some nonzero matrices $A_{ji} \in \mathbb{V}_{ji}$ and $B_{ki} \in
\mathbb{V}_{ki}$. Similar to before, the average of the two matrices
$\proj_{\mathbb{V}}\left( T I^{\left\{i\right\}} T^{\top} \right)$ and
$\proj_{\mathbb{V}}\left( U I^{\left\{i\right\}} U^{\top} \right)$ is
$I^{\left\{i\right\}} + \proj_{\mathbb{V}}\left( \hat{T}
I^{\left\{j,k\right\}} \hat{T}^{\top} \right)$, where the matrix $\hat{T}$
has
\[
\begin{aligned}
\hat{T}_{jj} &= \sqrt{\proj_{\mathbb{V}_{jj}}\left( A_{ji} A_{ji}^{\top}
\right)} = \sqrt{\mathcal{R}^{*}_{A_{ji}} \mathcal{R}_{A_{ji}} I_{n_{j}}}
\overset{\ref{identity:rightadjointself}}{=} \frac{1}{\sqrt{n_{j}}}
\left\lVert A_{ji}\right\rVert_{F} I_{n_{j}},
\\
\hat{T}_{kj} &= \sqrt{n_{j}} \left\lVert A_{ji}\right\rVert_{F}^{-1}
\proj_{\mathbb{V}_{kj}}\left( B_{ki} A_{ji}^{\top} \right) = \sqrt{n_{j}}
\left\lVert A_{ji}\right\rVert_{F}^{-1} \mathcal{R}^{*}_{A_{ji}} B_{ki},
\\
\hat{T}_{kk} &\overset{\ref{identity:rightadjointself}}{=}
\sqrt{\frac{1}{n_{k}} \left\lVert B_{ki}\right\rVert_{F}^{2} -
\frac{n_{j}}{n_{k}} \left\lVert A_{ji}\right\rVert_{F}^{-2} \left\lVert
\mathcal{R}^{*}_{A_{ji}} B_{ki}\right\rVert_{F}^{2}} I_{n_{k}},
\end{aligned}
\]
and zero blocks everywhere else. This matrix is a rank $3$ point if
\[
\left\lVert A_{ji}\right\rVert_{F}^{2} \left\lVert
B_{ki}\right\rVert_{F}^{2} > n_{j} \left\lVert \mathcal{R}^{*}_{A_{ji}}
B_{ki}\right\rVert_{F}^{2},
\]
in which case the matrix
\[
\sum_{l \in \left\{1, \dots, r\right\} \setminus \left\{j,k\right\}} I^{l}
+ \proj_{\mathbb{V}}\left( \hat{T} I^{\left\{j,k\right\}} \hat{T}^{\top}
\right)
\]
is in the relative interior $\proj_{\mathbb{V}} \mathbb{S}^{n}_{++}$ and is
also in the Minkowski sum of the $r-1$ extreme rays
$\mathsf{F}^{*}_{T,\left\{i\right\}}$,
$\mathsf{F}^{*}_{U,\left\{i\right\}}$, and
$\mathsf{F}^{*}_{I_{n},\left\{l\right\}}$ for $l \in \left\{1, \dots,
r\right\} \setminus \left\{i,j,k\right\}$, which shows that the
Carath\'eodory number of the dual cone $\bar{\mathsf{K}}^{*}$ is at most
$r-1$.

\begin{proposition}\label{proposition:Caratheodorynumbersnecessary}
For a homogeneous convex cone of rank $r$, if the Carath\'eodory number of
its closure is $r$, then, in any Ishi spectrahedral representation
$\mathbb{S}^{n}_{++} \cap \mathbb{V}$ of the homogeneous cone, for all $1
\leq i < j < k \leq r$, either of the following two equivalent conditions
hold:
\begin{enumerate}[label=(\roman*)]
\item For all $A_{ij} \in \mathbb{V}_{ij}$,
\[
\left\lVert A_{ij}\right\rVert_{F}^{2} \left\lVert
B_{ik}\right\rVert_{F}^{2} = n_{j} \left\lVert \mathcal{L}^{*}_{A_{ij}}
B_{ik}\right\rVert_{F}^{2} \quad \forall B_{ik} \in \mathbb{V}_{ik},
\]
or equivalently, $\mathcal{L}_{A_{ij}} \mathcal{L}^{*}_{A_{ij}} =
\frac{1}{n_{j}} \left\lVert A_{ij}\right\rVert_{F}^{2}
\mathcal{I}_{\mathbb{V}_{ik}}$.

\item $\mathbb{V}_{ik}$ and $\mathbb{V}_{jk}$ have the same dimensions
whenever $\mathbb{V}_{ij}$ is not the trivial subspace.

\end{enumerate}
Dually, if the Carath\'eodory number of its dual cone is $r$, then, in any
Ishi spectrahedral representation $\mathbb{S}^{n}_{++} \cap \mathbb{V}$ of
the homogeneous cone, for all $1 \leq k < j < i \leq r$, either of the
following two equivalent conditions hold:
\begin{enumerate}[label=(\roman*)]
\item For all $A_{ji} \in \mathbb{V}_{ji}$,
\[
\left\lVert A_{ji}\right\rVert_{F}^{2} \left\lVert
B_{ki}\right\rVert_{F}^{2} = n_{j} \left\lVert \mathcal{R}^{*}_{A_{ji}}
B_{ki}\right\rVert_{F}^{2} \quad \forall B_{ki} \in \mathbb{V}_{ki},
\]
or equivalently, $\mathcal{R}_{A_{ji}} \mathcal{R}^{*}_{A_{ji}} =
\frac{1}{n_{j}} \left\lVert A_{ji}\right\rVert_{F}^{2}
\mathcal{I}_{\mathbb{V}_{ki}}$.

\item $\mathbb{V}_{ki}$ and $\mathbb{V}_{kj}$ have the same dimensions
whenever $\mathbb{V}_{ji}$ is not the trivial subspace.

\end{enumerate}
\end{proposition}
\begin{proof}
In the preceding discussion, we have shown that a necessary condition, on
any Ishi spectrahedral representation $\mathsf{K} = \mathbb{S}^{n}_{++}
\cap \mathbb{V}$ of a homogeneous convex cone, for its closure and its dual
cone to have rank-matching Carath\'eodory numbers are, respectively,
\[
\forall 1 \leq i < j < k \leq r \ \forall (A_{ij}, B_{ik}) \in
\mathbb{V}_{ij} \times \mathbb{V}_{ik} \quad \left\lVert
A_{ij}\right\rVert_{F}^{2} \left\lVert B_{ik}\right\rVert_{F}^{2} = n_{j}
\left\lVert \mathcal{L}^{*}_{A_{ij}} B_{ik}\right\rVert_{F}^{2}
\]
and
\[
\forall 1 \leq k < j < i \leq r \ \forall (A_{ji}, B_{ki}) \in
\mathbb{V}_{ji} \times \mathbb{V}_{ki} \quad \left\lVert
A_{ji}\right\rVert_{F}^{2} \left\lVert B_{ki}\right\rVert_{F}^{2} = n_{j}
\left\lVert \mathcal{R}^{*}_{A_{ji}} B_{ki}\right\rVert_{F}^{2}.
\]

We further note that for every $(A_{ij}, B_{ik}) \in \mathbb{V}_{ij} \times
\mathbb{V}_{ik}$ with $1 \leq i < j < k \leq r$,
\[
\begin{split}
&\left\lVert \left\lVert A_{ij}\right\rVert_{F}^{2} B_{ik} - n_{j}
\mathcal{L}_{A_{ij}} \mathcal{L}^{*}_{A_{ij}} B_{ik}\right\rVert_{F}^{2}
\\
&=
\begin{aligned}[t]
&\left\lVert A_{ij}\right\rVert_{F}^{4} \left\lVert
B_{ik}\right\rVert_{F}^{2} - 2 n_{j} \left\lVert A_{ij}\right\rVert_{F}^{2}
\left\langle B_{ik} \middle| \mathcal{L}_{A_{ij}} \mathcal{L}^{*}_{A_{ij}}
B_{ik} \right\rangle_{\mathrm{F}}
\\
&{}+ n_{j}^2 \left\langle \mathcal{L}_{A_{ij}} \mathcal{L}^{*}_{A_{ij}}
B_{ik} \middle| \mathcal{L}_{A_{ij}} \mathcal{L}^{*}_{A_{ij}} B_{ik}
\right\rangle_{\mathrm{F}}
\end{aligned}
\\
&= \left\lVert A_{ij}\right\rVert_{F}^{4} \left\lVert
B_{ik}\right\rVert_{F}^{2} - 2 n_{j} \left\lVert A_{ij}\right\rVert_{F}^{2}
\left\lVert \mathcal{L}^{*}_{A_{ij}} B_{ik}\right\rVert_{F}^{2} + n_{j}^2
\left\langle \mathcal{L}^{*}_{A_{ij}} \mathcal{L}_{A_{ij}}
\mathcal{L}^{*}_{A_{ij}} B_{ik} \middle| \mathcal{L}^{*}_{A_{ij}} B_{ik}
\right\rangle_{\mathrm{F}}
\\
&\overset{\ref{identity:leftadjointself}}{=} \left\lVert
A_{ij}\right\rVert_{F}^{4} \left\lVert B_{ik}\right\rVert_{F}^{2} - 2 n_{j}
\left\lVert A_{ij}\right\rVert_{F}^{2} \left\lVert \mathcal{L}^{*}_{A_{ij}}
B_{ik}\right\rVert_{F}^{2} + n_{j} \left\lVert A_{ij}\right\rVert_{F}^{2}
\left\lVert \mathcal{L}^{*}_{A_{ij}} B_{ik}\right\rVert_{F}^{2}
\\
&= \left\lVert A_{ij}\right\rVert_{F}^{2} \left( \left\lVert
A_{ij}\right\rVert_{F}^{2} \left\lVert B_{ik}\right\rVert_{F}^{2} - n_{j}
\left\lVert \mathcal{L}^{*}_{A_{ij}} B_{ik}\right\rVert_{F}^{2} \right)
\end{split}
\]
shows that $\left\lVert A_{ij}\right\rVert_{F}^{2} \left\lVert
B_{ik}\right\rVert_{F}^{2} = n_{j} \left\lVert \mathcal{L}^{*}_{A_{ij}}
B_{ik}\right\rVert_{F}^{2}$ and $\left\lVert A_{ij}\right\rVert_{F}^{2}
B_{ik} = n_{j} \mathcal{L}_{A_{ij}} \mathcal{L}^{*}_{A_{ij}} B_{ik}$ are
equivalent. Similarly, for every $(A_{ji}, B_{ki}) \in \mathbb{V}_{ji}
\times \mathbb{V}_{ki}$ with $1 \leq k < j < i \leq r$,
\[
\left\lVert \left\lVert A_{ji}\right\rVert_{F}^{2} B_{ki} - n_{j}
\mathcal{R}_{A_{ji}} \mathcal{R}^{*}_{A_{ji}} B_{ki}\right\rVert_{F}^{2}
\overset{\ref{identity:rightadjointself}}{=} \left\lVert
A_{ji}\right\rVert_{F}^{2} \left( \left\lVert A_{ji}\right\rVert_{F}^{2}
\left\lVert B_{ki}\right\rVert_{F}^{2} - n_{j} \left\lVert
\mathcal{R}^{*}_{A_{ji}} B_{ki}\right\rVert_{F}^{2} \right)
\]
shows that $\left\lVert A_{ji}\right\rVert_{F}^{2} \left\lVert
B_{ki}\right\rVert_{F}^{2} = n_{j} \left\lVert \mathcal{R}^{*}_{A_{ji}}
B_{ki}\right\rVert_{F}^{2}$ and $\left\lVert A_{ji}\right\rVert_{F}^{2}
B_{ki} = n_{j} \mathcal{R}_{A_{ji}} \mathcal{R}^{*}_{A_{ji}} B_{ki}$ are
equivalent. Therefore, the above two necessary conditions are respectively
equivalent to
\[
\forall 1 \leq i < j < k \leq r \ \forall A_{ij} \in \mathbb{V}_{ij} \quad
\mathcal{L}_{A_{ij}} \mathcal{L}^{*}_{A_{ij}} = \frac{1}{n_{j}} \left\lVert
A_{ij}\right\rVert_{F}^{2} \mathcal{I}_{\mathbb{V}_{ik}}
\]
and
\[
\forall 1 \leq k < j < i \leq r \ \forall A_{ji} \in \mathbb{V}_{ji} \quad
\mathcal{R}_{A_{ji}} \mathcal{R}^{*}_{A_{ji}} = \frac{1}{n_{j}} \left\lVert
A_{ji}\right\rVert_{F}^{2} \mathcal{I}_{\mathbb{V}_{ki}},
\]
which trivially holds when, respectively, $A_{ij}$ and $A_{ji}$ are the
zero matrices. Otherwise when $A_{ij}$ (respectively, $A_{ji}$) is nonzero,
the identity \ref{identity:leftadjointself} (respectively,
\ref{identity:rightadjointself}) shows that the left multiplication
$\mathcal{L}_{A_{ij}}$ (respectively, right multiplication
$\mathcal{R}_{A_{ji}}$) is injective, and hence the necessary condition for
the closure (respectively, dual cone) is further equivalent to the linear
subspaces $\mathbb{V}_{ik}$ and $\mathbb{V}_{jk}$ (respectively,
$\mathbb{V}_{ki}$ and $\mathbb{V}_{kj}$) having the same dimensions for all
$k \in \left\{j+1, \dotsc, r\right\}$ (respectively, $k \in \left\{1,
\dotsc, i-1\right\}$).
\end{proof}

\subsection{A characterization for matching rank and Carath\'eodory number}
\label{section:Caratheodorynumberscharacterization}

In this section, we shall discover sufficient conditions for, respectively,
the closure of a homogeneous convex cone and its dual cone to have
rank-matching Carath\'eodory numbers, which then allows us to show that the
respective necessary conditions in the preceding section are in fact
sufficient.

\begin{proposition}\label{proposition:Caratheodorynumberssufficient}
For a homogeneous convex cone of rank $r$, if, in any Ishi spectrahedral
representation $\mathbb{S}^{n}_{++} \cap \mathbb{V}$ of the homogeneous
cone, for every $i \in \left\{1, \dotsc, r\right\}$, the orthogonal
projection of every extreme ray of the form
$\mathsf{F}_{T,\left\{i\right\}}$ onto the linear span of every face of the
form $\mathsf{F}_{I_{n},\mathsf{B}}$ with $\mathsf{B} \subseteq \left\{i+1,
\dotsc, r\right\}$ lies inside an extreme ray of the closure
$\mathbb{S}^{n}_{+} \cap \mathbb{V}$, then the Carath\'eodory number of
each nontrivial face of the closure matches the rank of its homogeneous
relative interior.

Dually, if, in any Ishi spectrahedral representation $\mathbb{S}^{n}_{++}
\cap \mathbb{V}$ of the homogeneous cone, for every $i \in \left\{1,
\dotsc, r\right\}$, the orthogonal projection of every extreme ray of the
form $\mathsf{F}^{*}_{T,\left\{i\right\}}$ onto the linear span of every
face of the form $\mathsf{F}^{*}_{I_{n},\mathsf{N}}$ with $\mathsf{N}
\subseteq \left\{1, \dotsc, i-1\right\}$ lies inside an extreme ray of the
dual cone $\proj_{\mathbb{V}} \mathbb{S}^{n}_{+}$, then the Carath\'eodory
number of each nontrivial face of the dual cone matches the rank of its
homogeneous relative interior.
\end{proposition}
\begin{proof}
We shall prove the contrapositive by supposing that there is a nontrivial
face of the closure $\bar{\mathsf{K}} = \mathbb{S}^{n}_{+} \cap \mathbb{V}$
whose Carath\'eodory number is less than the rank of its homogeneous
relative interior. Then, there must be a face $\mathsf{F}_{T,\mathsf{B}}$
of the closure that is minimal among nontrivial faces whose Carath\'eodory
numbers are less than the ranks of their respective relative interiors;
i.e., all proper faces of $\mathsf{F}_{T,\mathsf{B}}$ have rank-matching
Carath\'eodory numbers but $\mathsf{F}_{T,\mathsf{B}}$ does not. Since the
Carath\'eodory number of $\mathsf{F}_{T,\mathsf{B}}$ is at least that of
any of its proper faces, particularly that of its maximal proper faces, it
must be at least, and hence exactly, $\lvert \mathsf{B} \rvert - 1$. Of
course, $\lvert \mathsf{B} \rvert$ is at least $2$.\footnote{It is, in
fact, at least $3$, but this fact is not needed here.} Since the
Carath\'eodory numbers of faces and the ranks of their relative interiors
are invariant under the action of the transitive group of linear
automorphisms $\mathsf{T}_{\mathsf{K}}$ of $\mathsf{K}$, we may take the
face $\mathsf{F}_{T,\mathsf{B}}$ with $T = I_{n}$. We note that the face
$\mathsf{F}_{T,\mathsf{B}} = \mathsf{F}_{I_{n},\mathsf{B}}$ satisfies the
following three properties.
\begin{enumerate}[label=(\roman*)]
\item All faces of the form $\mathsf{F}_{I_{n},\left\{i, \dotsc, r\right\}
\cap \mathsf{B}}$ are invariant under the action of the simply transitive
group of linear automorphisms $\mathsf{T}_{\mathsf{F}_{I_{n},\mathsf{B}}}$
of the face $\mathsf{F}_{I_{n},\mathsf{B}}$ because $\mathbb{V}^{\left\{i,
\dotsc, r\right\} \cap \mathsf{B}}$ is invariant under the action of
$\mathsf{T}_{\mathsf{F}_{I_{n},\mathsf{B}}}$.

\item For each face of the form $\mathsf{F}_{I_{n},\left\{i, \dotsc,
r\right\} \cap \mathsf{B}}$, the orthogonal projection of the face
$\mathsf{F}_{I_{n},\mathsf{B}}$ onto its linear span $\mathbb{V}^{\left\{i,
\dotsc, r\right\} \cap \mathsf{B}}$ is the face
$\mathsf{F}_{I_{n},\left\{i, \dotsc, r\right\} \cap \mathsf{B}}$ itself,
because principal submatrices of positive semidefinite matrices are
positive semidefinite.

\item For every proper face of the form $\mathsf{F}_{I_{n},\left\{i,
\dotsc, r\right\} \cap \mathsf{B}}$ with $i \in \left\{\min \mathsf{B}+1,
\dotsc, r\right\}$ and every extreme ray $\mathsf{F}_{U,\left\{j\right\}}$
in $\mathsf{F}_{I_{n},\mathsf{B}} \setminus \mathsf{F}_{I_{n},\left\{i,
\dotsc, r\right\} \cap \mathsf{B}}$ (so that $j \in \mathsf{B} \setminus
\left\{i, \dotsc, r\right\}$), there is a linear automorphism of
$\mathsf{F}_{I_{n},\mathsf{B}}$ in
$\mathsf{T}_{\mathsf{F}_{I_{n},\mathsf{B}}}$ that maps the extreme ray
$\mathsf{F}_{U,\left\{j\right\}}$ to the extreme ray
$\mathsf{F}_{I_{n},\left\{j\right\}}$, which is orthogonal to the proper
face $\mathsf{F}_{I_{n},\left\{i, \dotsc, r\right\} \cap \mathsf{B}}$.

\end{enumerate}

Every matrix $X$ in the relative interior of the face
$\mathsf{F}_{I_{n},\mathsf{B}}$ is in the Minkowski sum of
$\kappa_{\mathsf{F}_{I_{n},\mathsf{B}}} = \lvert \mathsf{B} \rvert - 1$
extreme rays $\mathsf{R}(1), \dotsc, \mathsf{R}(\lvert \mathsf{B} \rvert -
1)$. Moreover, being in the relative interior of the face
$\mathsf{F}_{I_{n},\mathsf{B}}$, the matrix $X$ is not in the proper face
$\mathsf{F}_{I_{n},\mathsf{B}'}$ of $\mathsf{F}_{I_{n},\mathsf{B}}$, where
$\mathsf{B}' = \mathsf{B} \setminus \left\{\min \mathsf{B}\right\} =
\left\{\min \mathsf{B}+1, \dotsc, r\right\} \cap \mathsf{B}$, so that at
least one of these extreme rays, say, $\mathsf{R}(\lvert \mathsf{B} \rvert
- 1) = \mathsf{R}(\lvert \mathsf{B}' \rvert)$, must be outside the proper
face. By applying a linear automorphism of $\mathsf{F}_{I_{n},\mathsf{B}}$
in $\mathsf{T}_{\mathsf{F}_{I_{n},\mathsf{B}}}$ that maps the extreme ray
$\mathsf{R}(\lvert \mathsf{B}' \rvert)$ to an extreme ray orthogonal to the
proper face, we may pick a matrix $X$ in the face
$\mathsf{F}_{I_{n},\mathsf{B}}$ that is also in the Minkowski sum of the
extreme rays $\mathsf{R}(1), \dotsc, \mathsf{R}(\lvert \mathsf{B}' \rvert)$
with $\mathsf{R}(\lvert \mathsf{B}' \rvert)$ being orthogonal to the proper
face $\mathsf{F}_{I_{n},\mathsf{B}'}$.

So, the orthogonal projection of $X$ onto the linear span of the proper
face $\mathsf{F}_{I_{n},\mathsf{B}'}$ is the sum of the orthogonal
projections of the $\lvert \mathsf{B}' \rvert - 1$ extreme rays
$\mathsf{R}(1), \dotsc, \mathsf{R}(\lvert \mathsf{B}' \rvert - 1)$. Since
the face $\mathsf{F}_{I_{n},\mathsf{B}}$ projects orthogonally onto the
proper face $\mathsf{F}_{I_{n},\mathsf{B}'}$ and the matrix $X$ in the
relative interior of the face $\mathsf{F}_{I_{n},\mathsf{B}}$, the
orthogonal projection of $X$ must be in the relative interior of the proper
face $\mathsf{F}_{I_{n},\mathsf{B}'}$. This means that the sum of the
orthogonal projections of the $\lvert \mathsf{B}' \rvert - 1$ extreme rays
$\mathsf{R}(1), \dotsc, \mathsf{R}(\lvert \mathsf{B}' \rvert - 1)$ onto the
linear span of the proper face $\mathsf{F}_{I_{n},\mathsf{B}'}$ is in the
relative interior of the proper face.

Since all proper faces of $\mathsf{F}_{I_{n},\mathsf{B}'}$ have
rank-matching Carath\'eodory numbers, which are then less than $\lvert
\mathsf{B}' \rvert = \kappa_{\mathsf{F}_{I_{n},\mathsf{B}'}}$, the
Carath\'eodory number at matrices in the relative interior of
$\mathsf{F}_{I_{n},\mathsf{B}'}$ must be exactly
$\kappa_{\mathsf{F}_{I_{n},\mathsf{B}'}}$. So, no matrix in the relative
interior of the proper face $\mathsf{F}_{I_{n},\mathsf{B}'}$, including the
orthogonal projection of $X$, is in the Minkowski sum of at most
$\kappa_{\mathsf{F}_{I_{n},\mathsf{B}'}} - 1 = \lvert \mathsf{B}' \rvert -
1$ extreme rays, which means that at least one of the $\lvert \mathsf{B}'
\rvert - 1$ extreme rays $\mathsf{R}(1), \dotsc, \mathsf{R}(\lvert
\mathsf{B}' \rvert - 1)$ cannot project orthogonally into an extreme ray.
Of course, every extreme ray in the proper face
$\mathsf{F}_{I_{n},\mathsf{B}'}$ trivially projects orthogonally into
itself, which is an extreme ray. So, at least one of these extreme rays not
in the proper face $\mathsf{F}_{I_{n},\mathsf{B}'}$, which must be of the
form $\mathsf{F}_{T,\left\{\min \mathsf{B}\right\}}$, cannot project
orthogonally onto the linear span of the proper face
$\mathsf{F}_{I_{n},\mathsf{B}'}$ with $\mathsf{B}' \subseteq \left\{\min
\mathsf{B}+1, \dotsc, r\right\}$ to lie inside an extreme ray.

Similarly for the dual cone $\bar{\mathsf{K}}^{*} = \proj_{\mathbb{V}}
\mathbb{S}^{n}_{+}$ of the Ishi spectrahedral representation, we pick a
face $\mathsf{F}^{*}_{I_{n},\mathsf{N}}$ that is minimal among nontrivial
faces whose Carath\'eodory numbers are less than the ranks of their
respective relative interiors, which also satisfies the same three
properties with proper faces of the form $\mathsf{F}^{*}_{I_{n},\left\{1,
\dotsc, i\right\} \cap \mathsf{N}}$ and the simply transitive group of
linear automorphisms $\mathsf{T}_{\mathsf{F}^{*}_{I_{n},\mathsf{N}}}$. We
can then prove the dual statement with a similar argument.
\end{proof}

We now use these sufficient conditions to give the characterizations of,
respectively, homogeneous convex cones and their dual cones with
rank-matching Carath\'eodory numbers.

\begin{theorem}\label{theorem:characterizationofrankmatchingcaratheodorynumber}
For every homogeneous convex cone of rank $r$, the Carath\'eodory number of
its closure is $r$ if and only if, in any Ishi spectrahedral representation
$\mathsf{K} = \mathbb{S}^{n}_{++} \cap \mathbb{V}$ of the homogeneous
convex cone, for all $1 \leq i < j < k \leq r$, $\mathbb{V}_{ik}$ and
$\mathbb{V}_{jk}$ have the same dimensions whenever $\mathbb{V}_{ij}$ is
not the trivial subspace. In this case, the Carath\'eodory number of every
nontrivial face of its closure matches the rank of its homogeneous relative
interior.

Dually, the Carath\'eodory number of its dual cone is $r$ if and only if,
in any Ishi spectrahedral representation $\mathsf{K} = \mathbb{S}^{n}_{++}
\cap \mathbb{V}$ of the homogeneous convex cone, for all $1 \leq k < j < i
\leq r$, $\mathbb{V}_{ki}$ and $\mathbb{V}_{kj}$ have the same dimensions
whenever $\mathbb{V}_{ji}$ is not the trivial subspace. In this case, the
Carath\'eodory number of every nontrivial face of its dual cone matches the
rank of its homogeneous relative interior.
\end{theorem}
\begin{proof}
We have previously established in
Proposition~\ref{proposition:Caratheodorynumbersnecessary} the equivalence
of the respective dimension characterizations of rank-matching
Carath\'eodory numbers in the theorem statement to
\[
\forall 1 \leq i < j < k \leq r \ \forall A_{ij} \in \mathbb{V}_{ij} \quad
\mathcal{L}_{A_{ij}} \mathcal{L}^{*}_{A_{ij}} = \frac{1}{n_{j}} \left\lVert
A_{ij}\right\rVert_{F}^{2} \mathcal{I}_{\mathbb{V}_{ik}}
\]
and
\[
\forall 1 \leq k < j < i \leq r \ \forall A_{ji} \in \mathbb{V}_{ji} \quad
\mathcal{R}_{A_{ji}} \mathcal{R}^{*}_{A_{ji}} = \frac{1}{n_{j}} \left\lVert
A_{ji}\right\rVert_{F}^{2} \mathcal{I}_{\mathbb{V}_{ki}}.
\]
With the proofs of necessity in
Proposition~\ref{proposition:Caratheodorynumbersnecessary} and sufficiency
in Proposition~\ref{proposition:Caratheodorynumberssufficient}, it remains
to show that
\begin{enumerate}[label=(\roman*)]
\item for any $i \in \left\{1, \dotsc, r-1\right\}$, if
$\mathcal{L}_{A_{ij}} \mathcal{L}^{*}_{A_{ij}} = \frac{1}{n_{j}}
\left\lVert A_{ij}\right\rVert_{F}^{2} \mathcal{I}_{\mathbb{V}_{ik}}$ for
all $A_{ij} \in \mathbb{V}_{ij}$ with $j \in \left\{i+1, \dotsc, r\right\}$
and all $k \in \left\{j+1, \dotsc, r\right\}$, then the orthogonal
projection of any extreme ray of the form $\mathsf{F}_{T,\left\{i\right\}}$
onto the linear span of the face $\mathsf{F}_{I_{n},\mathsf{B}}$ with
$\mathsf{B} \subseteq \left\{i+1, \dotsc, r\right\}$ lies in an extreme
ray; and

\item for any $i \in \left\{2, \dotsc, r\right\}$, if $\mathcal{R}_{A_{ji}}
\mathcal{R}^{*}_{A_{ji}} = \frac{1}{n_{j}} \left\lVert
A_{ji}\right\rVert_{F}^{2} \mathcal{I}_{\mathbb{V}_{ki}}$ for all $A_{ji}
\in \mathbb{V}_{ji}$ with $j \in \left\{1, \dotsc, i-1\right\}$ and all $k
\in \left\{1, \dotsc, j-1\right\}$, then the orthogonal projection of any
extreme ray of the form $\mathsf{F}^{*}_{T,\left\{i\right\}}$ onto the
linear span of the face $\mathsf{F}^{*}_{I_{n},\mathsf{N}}$ with
$\mathsf{N} \subseteq \left\{1, \dotsc, i-1\right\}$ lies in an extreme
ray.

\end{enumerate}

To prove the first statement, we suppose that $\mathcal{L}_{A_{ij}}
\mathcal{L}^{*}_{A_{ij}} = \frac{1}{n_{j}} \left\lVert
A_{ij}\right\rVert_{F}^{2} \mathcal{I}_{\mathbb{V}_{ik}}$ for all $A_{ij}
\in \mathbb{V}_{ij}$ with $j \in \left\{i+1, \dotsc, r\right\}$ and all $k
\in \left\{j+1, \dotsc, r\right\}$ and that $\mathsf{B}$ is a nonempty
subset of $\left\{i+1, \dotsc, r\right\}$, and consider an extreme ray
$\mathsf{F}_{T,\left\{i\right\}} = \mathbb{R}_+(T^{\top}
I^{\left\{i\right\}} T)$. If the blocks $T_{ij}$ for $j \in \mathsf{B}$ are
all zero blocks, then $T^{\top} I^{\left\{i\right\}} T =
I^{\left\{i\right\}}$, and hence its orthogonal projection onto the linear
span of the face $\mathsf{F}_{I_{n},\mathsf{B}} = \mathbb{S}^{n}_{+} \cap
\mathbb{V}^{\mathsf{B}}$ is in the trivial face, which is in every extreme
ray.

Otherwise, we let $j \in \mathsf{B}$ denote the least index for which
$T_{ij}$ is a nonzero block. In this case, the orthogonal projection of
$T^{\top} I^{\left\{i\right\}} T$ is $\hat{T}^{\top} \hat{T}$, where
$\hat{T}$ has $\hat{T}_{ik} = T_{ik}$ for $k \in \left\{j, \dotsc,
r\right\} \cap \mathsf{B}$ and zero blocks everywhere else. We now show
that $\hat{T}^{\top} \hat{T} = \mathcal{T}_{U} I^{\left\{j\right\}} =
U^{\top} I^{\left\{j\right\}} U$, where $\mathcal{T}_{U} \in
\mathsf{T}_{\mathsf{K}}$ has $U_{jj} = \sqrt{n_{j}}^{-1} \left\lVert
T_{ij}\right\rVert_{F} I_{n_{j}}$, $U_{jk} = \sqrt{n_{j}} \left\lVert
T_{ij}\right\rVert_{F}^{-1} \mathcal{L}^{*}_{T_{ij}} T_{ik}$ for $k \in
\left\{j+1, \dotsc, r\right\} \cap \mathsf{B}$, and zero off-diagonal
blocks everywhere else. Indeed,
\begin{enumerate}[label=(\roman*)]
\item in the rows not indexed by $\left\{j, \dotsc, r\right\} \cap
\mathsf{B}$, both $\hat{T}^{\top} \hat{T}$ and $U^{\top}
I^{\left\{j\right\}} U$ have zero blocks;

\item the $j$-th diagonal blocks of $\hat{T}^{\top} \hat{T}$ and $U^{\top}
I^{\left\{j\right\}} U$ agree because
\[
\mathcal{L}^{*}_{U_{jj}} \mathcal{L}_{U_{jj}} I_{n_{j}} =
\frac{1}{n_{j}}\left\lVert T_{ij}\right\rVert_{F}^{2} I_{n_{j}}
\overset{\ref{identity:leftadjointself}}{=} \mathcal{L}^{*}_{T_{ij}}
\mathcal{L}_{T_{ij}} I_{n_{j}};
\]

\item for each $k \in \left\{j+1, \dotsc, r\right\} \cap \mathsf{B}$, the
$(j, k)$-th blocks of $\hat{T}^{\top} \hat{T}$ and $U^{\top}
I^{\left\{j\right\}} U$ agree because
\[
\mathcal{L}^{*}_{U_{jj}} U_{jk} = \sqrt{n_{j}}^{-1} \left\lVert
T_{ij}\right\rVert_{F} \sqrt{n_{j}} \left\lVert T_{ij}\right\rVert_{F}^{-1}
\mathcal{L}^{*}_{I_{n_{j}}} \mathcal{L}^{*}_{T_{ij}} T_{ik} =
\mathcal{L}^{*}_{T_{ij}} T_{ik};
\]
and, finally,

\item for any $k \in \left\{j+1, \dotsc, r\right\} \cap \mathsf{B}$ and any
$l \in \left\{j+1, \dotsc, k\right\} \cap \mathsf{B}$, the $(l, k)$-th
blocks of $\hat{T}^{\top} \hat{T}$ and $U^{\top} I^{\left\{j\right\}} U$
agree because
\[
\begin{split}
\mathcal{L}^{*}_{U_{jl}} U_{jk}
&= n_{j} \left\lVert T_{ij}\right\rVert_{F}^{-2}
\mathcal{L}^{*}_{\mathcal{L}^{*}_{T_{ij}} T_{il}} \mathcal{L}^{*}_{T_{ij}}
T_{ik}
\\
&\overset{\ref{identity:leftadjointleft}}{=} n_{j} \left\lVert
T_{ij}\right\rVert_{F}^{-2} \mathcal{L}^{*}_{T_{il}} \mathcal{L}_{T_{ij}}
\mathcal{L}^{*}_{T_{ij}} T_{ik}
\\
&= \mathcal{L}^{*}_{T_{il}} \mathcal{I}_{\mathbb{V}_{ik}} T_{ik}
= \mathcal{L}^{*}_{T_{il}} T_{ik},
\end{split}
\]
where the third equality follows from the hypothesis $\mathcal{L}_{A_{ij}}
\mathcal{L}^{*}_{A_{ij}} = \frac{1}{n_{j}} \left\lVert
A_{ij}\right\rVert_{F}^{2} \mathcal{I}_{\mathbb{V}_{ik}}$.

\end{enumerate}
Therefore, the orthogonal projection of $T^{\top} I^{\left\{i\right\}} T$,
and hence of the face $\mathsf{F}_{T,\left\{i\right\}}$, onto the linear
span of the face $\mathsf{F}_{I_{n},\mathsf{B}}$ is the extreme ray
$\mathsf{F}_{U,\left\{j\right\}}$.

The proof of the statement for the dual cone is similar, where, in the
nontrivial case, the orthogonal projection of $\proj_{\mathbb{V}}\left( T
I^{\left\{i\right\}} T^{\top} \right)$ is $\proj_{\mathbb{V}}\left( \hat{T}
\hat{T}^{\top} \right)$ with $\hat{T}$ having $\hat{T}_{ki} = T_{ki}$ for
$k \in \left\{1, \dotsc, j\right\} \cap \mathsf{N}$ and zero blocks
everywhere else; which can then be shown to be the same as
$\mathcal{T}^{*}_{U} I^{\left\{j\right\}} = \proj_{\mathbb{V}}\left( U
I^{\left\{j\right\}} U^{\top} \right)$, with $U \in \mathbb{T}$ having
$U_{jj} = \sqrt{n_{j}}^{-1} \left\lVert T_{ji}\right\rVert_{F} I_{n_{j}}$,
$U_{kj} = \sqrt{n_{j}} \left\lVert T_{ji}\right\rVert_{F}^{-1}
\mathcal{R}^{*}_{T_{ji}} T_{ki}$ for $k \in \left\{1, \dotsc, j-1\right\}
\cap \mathsf{N}$, and zero off-diagonal blocks everywhere else, using a
similar argument with the left multiplication operators replaced by
corresponding right multiplication operators.
\end{proof}

\section{Applications to selfdual homogeneous convex cones and sparse
spectrahedral cones}
\label{section:applications}

We now apply the above characterization of homogeneous convex cones with
rank-matching Carath\'eodory numbers to selfdual homogeneous convex cones
and sparse spectrahedral cones.

\subsection{Characterization of selfdual homogeneous convex cones}
\label{section:selfdualhomogeneouscones}

In this section, we shall show that among homogeneous convex cones, the
selfdual cones are precisely those whose ranks match the respective
Carath\'eodory numbers of both their closures and their dual cones.

We first consider Ishi spectrahedral representations $\mathsf{K} =
\mathbb{S}^{n}_{++} \cap \mathbb{V}$ of indecomposable homogeneous convex
cones of rank $r$. An indecomposable convex cone is one that cannot be
expressed as the direct sum\footnote{A direct sum is a Minkowski sum of
sets whose linear spans have trivial pairwise intersections.} $\mathsf{K}_1
\oplus \mathsf{K}_2$ of two nontrivial convex cones. The indecomposability
means that the linear subspace $\mathbb{V}$ cannot be decomposed into
$\mathbb{V} = \mathbb{V}_{\mathsf{I}} \oplus \mathbb{V}_{\mathsf{J}}$,
where $\mathsf{I} \cup \mathsf{J}$ is a partition of $\left\{1, \dotsc,
r\right\}$, and $\mathbb{V}_{\mathsf{I}}$ and $\mathbb{V}_{\mathsf{J}}$ are
similarly defined as $\mathbb{V}_{\mathsf{B}}$.

Theorem~\ref{theorem:characterizationofrankmatchingcaratheodorynumber}
implies that for the rank of an indecomposable Ishi spectrahedral
representation $\mathsf{K}$ to match the Carath\'eodory numbers of both its
closure $\mathbb{S}^{n}_{+} \cap \mathbb{V}$ and its dual cone
$\bar{\mathsf{K}}^{*} = \proj_{\mathbb{V}} \mathbb{S}^{n}_{+}$, it is both
necessary and sufficient that whenever $\mathbb{V}_{ij}$ with $1 \leq i < j
\leq r$ is not the trivial subspace, the linear subspaces $\mathbb{V}_{ik}$
and $\mathbb{V}_{jk}$ have the same dimensions for $k \in \left\{j+1,
\dotsc, r\right\}$ and so do the linear subspaces $\mathbb{V}_{ki}$ and
$\mathbb{V}_{kj}$ for $k \in \left\{1, \dotsc, i-1\right\}$. This necessary
and sufficient condition trivially holds when all linear subspaces
$\mathbb{V}_{ij}$ have the same dimensions. On the other hand, we shall now
show that if this necessary and sufficient condition holds, then, because
of the indecomposability, none of the linear subspaces $\mathbb{V}_{ij}$ is
the trivial subspace, and hence all of them have the same dimensions.

We shall prove the contrapositive by supposing that $\mathbb{V}_{ij}$ is
the trivial subspace for some $1 \leq i < j \leq r$. We consider the
partition $\mathsf{I} \cup \mathsf{J}$ of $\left\{1, \dotsc, r\right\}$
with
\[
\mathsf{J} =
\begin{aligned}[t]
&\left\{k \in \left\{1, \dotsc, i-1\right\}\ \middle|\ \mathbb{V}_{ki}
\text{ is the trivial subspace}\right\}
\\
&\cup \left\{k \in \left\{i+1, \dotsc, r\right\}\ \middle|\ \mathbb{V}_{ik}
\text{ is the trivial subspace}\right\}.
\end{aligned}
\]
We note that $i \notin \mathsf{J}$ and $j \in \mathsf{J}$ mean that
$\mathsf{I}$ and $\mathsf{J}$ are respectively nonempty. Since $\mathbb{V}$
is indecomposable, there must be some $k \in \mathsf{I}$ and some $l \in
\mathsf{J}$ for which $\mathbb{V}_{kl}$, or $\mathbb{V}_{lk}$ in the case
$l < k$, is not the trivial subspace. We note that $k \neq i$ because $l
\in \mathsf{J}$ means that $\mathbb{V}_{il}$, or $\mathbb{V}_{li}$ in the
case $l < i$, is the trivial subspace.
\begin{enumerate}[label=(\roman*)]
\item If $k$ and $l$ are both less than $i$, the linear subspace
$\mathbb{V}_{kl}$, or $\mathbb{V}_{lk}$ in the case $l < k$, is not the
trivial subspace, and yet $\mathbb{V}_{ki}$ and $\mathbb{V}_{li}$ have
different dimensions.

\item If $k$ and $l$ are both greater than $i$, the linear subspace
$\mathbb{V}_{kl}$, or $\mathbb{V}_{lk}$ in the case $l < k$, is not the
trivial subspace, and yet $\mathbb{V}_{ik}$ and $\mathbb{V}_{il}$ have
different dimensions.

\item If $i$ lies between $k$ and $l$, the linear subspace
$\mathbb{V}_{ki}$, or $\mathbb{V}_{ik}$ in the case $l < k$, is not the
trivial subspace, and yet $\mathbb{V}_{kl}$ and $\mathbb{V}_{il}$, or
$\mathbb{V}_{lk}$ and $\mathbb{V}_{li}$ in the case $l < k$, have different
dimensions.

\end{enumerate}

When the linear subspaces $\mathbb{V}_{ij}$ in the Ishi spectrahedral
representation $\mathbb{S}^{n}_{++} \cap \mathbb{V}$ have the same
dimensions, it was shown by Vinberg \cite[Proposition~II.2.3]{Vin65}, using
$T$-algebraic representations of homogeneous convex cones, that the
indecomposable homogeneous convex cone is selfdual\footnote{Every Ishi
spectrahedral representation and every $T$-algebraic representation of a
homogeneous convex cone are related, via a compact normal left symmetric
algebra constructed from a maximal triangular linear automorphism group of
the cone, in such a way that the dimension of each linear subspace
$\mathbb{V}_{ij}$ in the Ishi spectrahedral representation coincides with
that of the corresponding subspace in the $T$-algebra representation.}. On
the other hand, the Ishi spectrahedral representation $\mathbb{S}^{n}_{++}
\cap \mathbb{V}$ of every indecomposable selfdual homogeneous convex cone
has a constant dimension for the linear subspaces
$\mathbb{V}_{ij}$.\footnote{See \cite[\S{II}.2]{Vin65} and the preceding
footnote.} Therefore, every indecomposable homogeneous convex cone is
selfdual if and only if both its closure and its dual cone have
rank-matching Carath\'eodory numbers.

When the homogeneous convex cone is not necessarily indecomposable, it can
be uniquely expressed (up to linear isomorphism) as the direct sum
$\mathsf{K}_{1} \oplus \dotsb \oplus \mathsf{K}_{p}$ of nontrivial
indecomposable convex cones, which are its indecomposable components; see
\cite[\S4]{Hor78}. At the same time, its dual cone is linearly isomorphic
to the direct sum of the dual cones of the indecomposable components. This
can be seen by first making the linear spans of the indecomposable
components pairwise orthogonal via a change of basis, followed by taking
the intersection of the dual cones of the resulting indecomposable
components; see \cite[Corollary~16.4.2]{Rock97}. Moreover, we shall show
that the indecomposable components, and hence the relative interiors of
their dual cones, are homogeneous in the following proposition.
\begin{proposition}
Every homogeneous convex cone $\mathsf{K}$ can be uniquely expressed (up to
linear isomorphism) as the direct sum $\mathsf{K}_{1} \oplus \dotsb \oplus
\mathsf{K}_{p}$ of indecomposable components, each of which is a
homogeneous convex cone. Its dual cone $\bar{\mathsf{K}}^{*}$, which has a
homogeneous relative interior, is then linearly isomorphic to the direct
sum $\bar{\mathsf{K}}_{1}^{*} \oplus \dotsb \oplus
\bar{\mathsf{K}}_{p}^{*}$ of the dual cones of the respective
indecomposable components, each of which has a homogeneous relative
interior.
\end{proposition}
\begin{proof}
We first consider the case where the indecomposable components are pairwise
linearly isomorphic, with $\mathcal{H}_{i} : \mathsf{K}_{1} \to
\mathsf{K}_{i}$ denoting a linear isomorphism for each $i \in \{1, \dotsc,
p\}$. Horne \cite[Lemma~4.1]{Hor78} showed that the linear automorphism
group of $\mathsf{K}$ is the group of linear extensions of
\[
\sum_{i=1}^{p} x_{i} \in \bigoplus_{i=1}^{p} \mathsf{K}_{i} \mapsto
\sum_{i=1}^{p} \mathcal{H}_{\sigma(i)} \mathcal{T}_{i} \mathcal{H}_{i}^{-1}
x_{i},
\]
over all $(\sigma, \mathcal{T}_{1}, \dotsc, \mathcal{T}_{p})$ where
$\sigma$ is a permutation of the set $\{1, \dotsc, p\}$ and
$\mathcal{T}_{1}, \dotsc, \mathcal{T}_{p}$ are linear automorphisms of
$\mathsf{K}_{1}$. We denote this linear extension by $\mathcal{T}_{\sigma,
\mathcal{T}_{1}, \dotsc, \mathcal{T}_{p}}$. For each pair of points $x, y
\in \mathsf{K}_{1}$, the two points
\[
w = \sum_{i=1}^{p} \mathcal{H}_{i} x \quad \text{and} \quad z =
\sum_{i=1}^{p} \mathcal{H}_{i} y
\]
are in $\mathsf{K}$, which is homogeneous. Therefore, there is a linear
automorphism $\mathcal{T}_{\sigma, \mathcal{T}_{1}, \dotsc,
\mathcal{T}_{p}}$ mapping $w$ to $z$, which is equivalent to every one of
the linear automorphisms $\mathcal{T}_{i}$ mapping $x$ to $y$. This shows
that $\mathsf{K}_{1}$ is homogeneous.

We now turn to the general case, where we gather the linearly isomorphic
indecomposable components to express the homogeneous convex cone
$\mathsf{K}$ as the direct sum $\hat{\mathsf{K}}_{1} \oplus \dotsb \oplus
\hat{\mathsf{K}}_{q}$ of convex cones, with each convex cone
$\hat{\mathsf{K}}_{j}$ the direct sum $\hat{\mathsf{K}}_{j1} \oplus \dotsb
\oplus \hat{\mathsf{K}}_{jp_j}$ of pairwise linearly isomorphic
indecomposable convex cones (so that $p = p_1 + \dotsb + p_q$). Moreover,
this is done in such a way that the indecomposable components of different
$\hat{\mathsf{K}}_{j}$'s are not linearly isomorphic. Horne
\cite[Theorem~4.2]{Hor78} showed that the linear automorphism group of
$\mathsf{K}$ is the group of linear extensions of
\[
\sum_{j=1}^{q} x_{j} \in \bigoplus_{j=1}^{q} \hat{\mathsf{K}}_{j} \mapsto
\sum_{j=1}^{q} \mathcal{T}_{j} x_{j},
\]
over all $(\mathcal{T}_{1}, \dotsc, \mathcal{T}_{q})$ where each
$\mathcal{T}_{j}$ is a linear automorphism of the respective convex cone
$\hat{\mathsf{K}}_{j}$. We denote this linear extension by
$\mathcal{T}_{\mathcal{T}_{1}, \dotsc, \mathcal{T}_{q}}$. For each convex
cone $\hat{\mathsf{K}}_{j}$, we pick an arbitrary but fixed point $x_{i}
\in \hat{\mathsf{K}}_{i}$ for each $i \in \{1, \dotsc, q\} \setminus
\{j\}$, so that for each pair of points $x, y \in \hat{\mathsf{K}}_{j}$,
the two points
\[
w = x + \sum_{i \in \{1, \dotsc, q\} \setminus \{j\}} x_{i} \quad
\text{and} \quad z = y + \sum_{i \in \{1, \dotsc, q\} \setminus \{j\}}
x_{i}
\]
are in $\mathsf{K}$, which is homogeneous. Therefore, there is a linear
automorphism $\mathcal{T}_{\mathcal{T}_{1}, \dotsc, \mathcal{T}_{q}}$
mapping $w$ to $z$, which means that the linear automorphism
$\mathcal{T}_{j}$ maps $x$ to $y$. This shows that the convex cone
$\hat{\mathsf{K}}_{j}$ is homogeneous. Consequently, each indecomposable
component of each homogeneous convex cone $\hat{\mathsf{K}}_{j}$ is
homogeneous.
\end{proof}
When a homogenous convex cone is expressed as the direct sum of
homogeneous convex cones, the Carath\'eodory number of its closure, and its
rank, are related to those of the component cones as follows.
\begin{proposition}
If a convex cone $\mathsf{K}$ is the direct sum $\mathsf{K}_{1} \oplus
\mathsf{K}_{2}$ of two nontrivial convex cones, then faces of its closure
$\bar{\mathsf{K}}$ are precisely direct sums of faces of the respective
closures $\bar{\mathsf{K}}_{1}$ and $\bar{\mathsf{K}}_{2}$ of the convex
cones $\mathsf{K}_{1}$ and $\mathsf{K}_{2}$.

Moreover, the Carath\'eodory number of the closure $\bar{\mathsf{K}}$ is
the sum of the Carath\'eodory numbers of the closures
$\bar{\mathsf{K}}_{1}$ and $\bar{\mathsf{K}}_{2}$.

If, in addition, the convex cones $\mathsf{K}$, $\mathsf{K}_{1}$, and
$\mathsf{K}_{2}$ are homogeneous, then the rank of the homogeneous convex
cone $\mathsf{K}$ is the sum of the ranks of the homogenous convex cones
$\mathsf{K}_{1}$ and $\mathsf{K}_{2}$.
\end{proposition}
\begin{proof}
In \cite[Lemma~3.2]{LouNi26}, the authors proved that a convex subset of
the Minkowski sum of the two closed convex cones is a face only if it is
the Minkowski sum of faces of the respective convex cones. When the
Minkowski sum is a direct sum, the converse follows directly from the
definition of faces.

In particular, extreme rays of the closure $\bar{\mathsf{K}}$ are precisely
extreme rays of the closures $\bar{\mathsf{K}}_{1}$ and
$\bar{\mathsf{K}}_{2}$. Therefore, every point $x = x_{1} + x_{2}$ in the
closure $\bar{\mathsf{K}}$, with $(x_{1}, x_{2}) \in \bar{\mathsf{K}_{1}}
\times \bar{\mathsf{K}_{2}}$, is in the Minkowski sum of $\kappa$ extreme
rays of the closure $\bar{\mathsf{K}}$ if and only if, for some
$\kappa_{1}$ and $\kappa_{2}$ summing up to $\kappa$, each $x_{i}$ is in
the Minkowski sum of $\kappa_i$ extreme rays of the closure
$\bar{\mathsf{K}}_{i}$.

The above characterization of faces of the closure $\bar{\mathsf{K}}$ as
direct sum of faces also means that if
\[
\mathsf{F}_{1}(1) \subsetneqq \dotsb \subsetneqq \mathsf{F}_{1}(r_{1})\ (=
\bar{\mathsf{K}}_{1}) \quad \text{and} \quad \mathsf{F}_{2}(1) \subsetneqq
\dotsb \subsetneqq \mathsf{F}_{2}(r_{2})\ (= \bar{\mathsf{K}}_{2})
\]
are maximal chains of proper faces of the respective closures
$\bar{\mathsf{K}}_{1}$ and $\bar{\mathsf{K}}_{2}$, then
\[
\mathsf{F}_{1}(1) \subsetneqq \dotsb \subsetneqq \mathsf{F}_{1}(r_{1})
\subsetneqq \mathsf{F}_{1}(r_{1}) \oplus \mathsf{F}_{2}(1) \subsetneqq
\dotsb \subsetneqq \mathsf{F}_{1}(r_{1}) \oplus \mathsf{F}_{2}(r_2)\ (=
\bar{\mathsf{K}}_{1} \oplus \bar{\mathsf{K}}_{2} = \bar{\mathsf{K}})
\]
is a maximal chain of proper faces of the closure $\bar{\mathsf{K}}$. When
the convex cones $\mathsf{K}$, $\mathsf{K}_{1}$, and $\mathsf{K}_{2}$ are,
in addition, homogeneous, the geometrical characterization of the rank of a
homogeneous convex cone, as the length of a maximal chain of proper faces,
then gives the ranks of the homogeneous convex cones $\mathsf{K}$,
$\mathsf{K}_{1}$, and $\mathsf{K}_{2}$ as $r_{1} + r_{2}$, $r_{1}$, and
$r_{2}$, respectively.
\end{proof}

With these two propositions, we may apply the above argument to each of
the indecomposable components of a homogeneous convex cone to conclude that

\begin{theorem}\label{theorem:selfdualhomogeneouscone}
The rank of a homogeneous convex cone matches the Carath\'eodory numbers of
both its closure and its dual cone if and only if it is the direct sum of
indecomposable selfdual homogeneous convex cones; i.e., it is a selfdual
homogeneous convex cone.
\end{theorem}

\subsection{Characterization of homogeneous sparse spectrahedral cones}
\label{section:sparsespectrahedralcones}

A block matrix spectrahedral cone $\mathbb{S}^{n}_{++} \cap \mathbb{V}$
with $n_1 = \dotsb = n_r = 1$ is the cone of real symmetric positive
definite matrices in $\mathbb{S}^{n}_{++}$ with a prescribed sparsity
pattern, where the dimensions of the linear subspaces $\mathbb{V}_{ij}$
specify the zero entries. We shall call this a \emph{sparse spectrahedral
cone}. For a sparse spectrahedral cone $\mathbb{S}^{n}_{++} \cap
\mathbb{V}$, we denote by $\hat{E}$ the set of pairs of indices $(i, j) \in
\left\{1, \dotsc, n\right\}^{2}$ for which $i < j$ and $\mathbb{V}_{ij} =
\mathbb{R}$. In terms of the set $\hat{E}$, the property
\ref{axiom:triangularproductblock} is
\begin{equation}
\label{property:trivialperfectelimination}
(i, k) \in \hat{E} \quad \forall (i, j), (j, k) \in \hat{E},
\end{equation}
and the property \ref{axiom:mixedtriangularproductblock} is
\begin{equation}
\label{property:perfectelimination}
(j, k) \in \hat{E} \quad \forall (i, j), (i, k) \in \hat{E} \text{ with } j
< k.
\end{equation}
Of course, the property \ref{axiom:selfproductblock} is immediately
satisfied because $n_{1} = \dotsb = n_{r} = 1$. So, a sparse spectrahedral
cone is also a homogeneous convex cone when it satisfies the above two
properties \eqref{property:trivialperfectelimination} and
\eqref{property:perfectelimination}. It is natural to ask if, up to linear
isomorphism, the two properties  \eqref{property:trivialperfectelimination}
and \eqref{property:perfectelimination} are also necessary for a sparse
spectrahedral cone to be a homogeneous convex cone.

Every sparse spectrahedral cone can be combinatorially described by a
simple undirected graph $G = (V, E)$ with $n = \lvert V \rvert$ vertices
labelled $1, \dotsc, n$ and edge set $E$ containing the doubletons
$\left\{i, j\right\} \subseteq \left\{1, \dotsc, n\right\}$ for which $i
\neq j$ and the $(i, j)$-th entry is allowed to be nonzero; i.e., $E =
\left\{\left\{i, j\right\}\ \middle|\ (i, j) \in \hat{E}\right\}$. This
graph describes the sparse spectrahedral cone as $\mathsf{K}(G) :=
\mathbb{S}^{n}_{++} \cap \mathbb{V}(G)$, where
\[
\mathbb{V}(G) := \left\{X \in \mathbb{S}^{n}\ \middle|\ X_{ij} = 0 \quad
\forall i \neq j, \{i, j\} \notin E\right\}.
\]
Of course, isomorphic graphs describe linearly isomorphic sparse
spectrahedral cones.

So, a sparse spectrahedral cone $\mathsf{K} = \mathbb{S}^{n}_{++} \cap
\mathbb{V}$, which can be described by a graph $G$ with edge set $E =
\left\{\left\{i, j\right\}\ \middle|\ (i, j) \in \hat{E}\right\}$, is a
homogeneous convex cone if the graph $G$ satisfies similar properties
induced by \eqref{property:trivialperfectelimination} and
\eqref{property:perfectelimination}. Simple undirected graphs with a
suitable labelling of the vertices satisfying the property induced by
\eqref{property:perfectelimination} are said to be chordal. These graphs
are also characterized by not having cycles $C_{k}$ of length $k \geq 4$ as
induced subgraphs on $k$ vertices. Simple undirected graphs with a suitable
labelling of the vertices satisfying, in addition, the property induced by
\eqref{property:trivialperfectelimination} are said to be homogeneous
chordal. These graphs are further characterized by not having $4$-paths
$P_{4}$ as induced subgraphs on four vertices. This means that they are
characterized by the absence of $P_{4}$ and $C_{4}$ as induced subgraphs
because every induced cycle of length $k > 4$ necessarily contains an
induced 4-path. For further details, see \cite[\S2 and \S3]{TunVan23}.

When a sparse spectrahedral cone $\mathsf{K}(G) = \mathbb{S}^{n}_{++} \cap
\mathbb{V}$, where $\mathbb{V} = \mathbb{V}(G)$, is described by a simple
undirected graph $G$ and $H$ is an induced subgraph of $G$ induced by the
vertices $i_{1}, \dotsc, i_{p}$, the face $\mathsf{F}_{I_{n},\left\{i_{1},
\dotsc, i_{p}\right\}} = \mathbb{S}^{n}_{+} \cap \mathbb{V}^{\left\{i_{1},
\dotsc, i_{p}\right\}}$ of the closure $\bar{\mathsf{K}}(G) =
\mathbb{S}^{n}_{+} \cap \mathbb{V}$ is linearly isomorphic to the sparse
spectrahedral cone $\mathsf{K}(H) = \mathbb{S}^{p}_{++} \cap \mathbb{V}(H)$
described by the graph $H$; see
Section~\ref{section:facialstructureprimal}. Therefore, since faces of
closures of homogeneous convex cones have homogeneous relative interiors,
if the sparse spectrahedral cone $\mathsf{K}(G)$ is also a homogeneous
convex cone, then so is the sparse spectrahedral cone $\mathsf{K}(H)$
described by any induced subgraph $H$ of $G$. This result was previously
proved by Ishi \cite[Theorem~7]{Ishi13} using compact normal left symmetric
algebras associated with homogeneous convex cones; cf.\@
\cite[Theorem~4.2]{GouItoLou25}.

We shall use the characterization of homogeneous convex cones with
rank-matching Carath\'eodory numbers in
Theorem~\ref{theorem:characterizationofrankmatchingcaratheodorynumber} to
show that neither of $\mathsf{K}(P_{4})$ nor $\mathsf{K}(C_{4})$, which are
indecomposable sparse spectrahedral cones of order $4$, is a homogeneous
convex cone, and hence any graph $G$ that describes a sparse spectrahedral
cone $\mathsf{K}(G)$ that is also a homogeneous convex cone must be
homogeneous chordal. This characterization of homogeneous sparse
spectrahedral cones was stated without proof in
\cite[Theorem~2.2]{LetMas07}, and subsequently proved by Ishi
\cite[Theorem~A]{Ishi13}, who showed that the sparse spectrahedral cones
$\mathsf{K}(P_{4})$ and $\mathsf{K}(C_{4})$ cannot be associated with any
compact normal left symmetric algebra. Our geometrical approach here uses
the Carath\'eodory numbers and the facial structure of sparse spectrahedral
cones, which complements the algebraic approach by Ishi. Let us begin with
a note on the Carath\'eodory numbers of sparse spectrahedral cones that are
also homogeneous convex cones, which, as far as we know, is a new result.

\begin{theorem}\label{theorem:sparsespectrahedralcone}
If a sparse spectrahedral cone $\mathbb{S}^{n}_{++} \cap \mathbb{V}$ is a
homogeneous convex cone, then its rank and the Carath\'eodory number of its
closure $\mathbb{S}^{n}_{+} \cap \mathbb{V}$ both coincide with the order
$n$.
\end{theorem}
\begin{proof}
When a sparse spectrahedral cone $\mathbb{S}^{n}_{++} \cap \mathbb{V}$ is a
homogeneous convex cone of rank $r$, the maximal chain of proper faces
$\mathbb{S}^{n}_{+} \cap \mathbb{V}^{\left\{1, \dotsc, n\right\}}, \dotsc,
\mathbb{S}^{n}_{+} \cap \mathbb{V}^{\left\{1\right\}}$ of the closure
$\mathbb{S}^{n}_{+} \cap \mathbb{V}$ shows that its rank $r$ matches the
order $n$; see Section~\ref{section:facialstructureprimal}. Moreover, every
maximal chain of proper faces $\mathsf{F}(1), \dotsc, \mathsf{F}(r = n)$ of
its closure $\mathbb{S}^{n}_{+} \cap \mathbb{V}$ is of the form
$\mathsf{F}'(1) \cap \mathbb{V}, \dotsc, \mathsf{F}'(n) \cap \mathbb{V}$
for some chain of proper faces $\mathsf{F}'(1), \dotsc, \mathsf{F}'(n)$ of
$\mathbb{S}^{n}_{+}$ (see Proposition~\ref{proposition:facesofslices}),
which must then be also maximal. Therefore, we conclude that the rank of
the homogeneous relative interior of each face $\mathsf{F}(p)$ and the
order of the respective face $\mathsf{F}'(p)$ (i.e., the order of the
positive semidefinite cone to which $\mathsf{F}'(p)$ is linearly
isomorphic) both coincide with $n-p+1$. Since every proper face
$\mathsf{F}$ of the closure $\mathbb{S}^{n}_{+} \cap \mathbb{V}$ belongs to
some maximal chain of proper faces, the rank of its homogeneous relative
interior must then match the order of the proper face $\mathsf{F}'$ of
$\mathbb{S}^{n}_{+}$ for which $\mathsf{F} = \mathsf{F}' \cap \mathbb{V}$.
In particular, every extreme ray of the closure $\mathbb{S}^{n}_{+} \cap
\mathbb{V}$, whose homogeneous relative interior has rank $1$, must be an
extreme ray of $\mathbb{S}^{n}_{+}$ because it is of the form $\mathsf{F}'
\cap \mathbb{V}$ for some face $\mathsf{F}'$ of $\mathbb{S}^{n}_{+}$ of
order 1. So, the matrix $I_{n}$ in both homogeneous convex cones
$\mathbb{S}^{n}_{++}$ and $\mathbb{S}^{n}_{++} \cap \mathbb{V}$ is in the
Minkowski sum of $\kappa_{\mathbb{S}^{n}_{+} \cap \mathbb{V}}$ extreme rays
of the closure $\mathbb{S}^{n}_{+}$. This shows that $r \geq
\kappa_{\mathbb{S}^{n}_{+} \cap \mathbb{V}} \geq
\kappa_{\mathbb{S}^{n}_{+}} = n$.
\end{proof}

As a consequence of the above theorem, a sparse spectrahedral cone that is
also a homogeneous convex cone must have rank $r = n$ and every Ishi
spectrahedral representation $\mathsf{K} = \mathbb{S}^{m}_{++} \cap
\mathbb{W}$, with $m \geq r = n$, satisfies the necessary and sufficient
condition given in
Theorem~\ref{theorem:characterizationofrankmatchingcaratheodorynumber} for
its rank $r$ to match the Carath\'eodory number $\kappa_{\bar{\mathsf{K}}}$
of its closure $\bar{\mathsf{K}} = \mathbb{S}^{m}_{+} \cap \mathbb{W}$;
namely, for every $1 \leq i < j < k \leq r$, the linear subspaces
$\mathbb{W}_{ik}$ and $\mathbb{W}_{jk}$ have the same dimensions whenever
$\mathbb{W}_{ij}$ is not the trivial subspace. We further note that when
$\mathbb{S}^{m}_{++} \cap \mathbb{W}$ is linearly isomorphic to a sparse
spectrahedral cone, every linear subspace $\mathbb{W}_{ij}$ has dimension
at most one. This is because, as shown in the above proof, every face of
the closure $\bar{\mathsf{K}}$ has a spectrahedral representation whose
order matches the rank of its homogeneous relative interior, which means
the face $\mathsf{F}_{I_{m},\left\{i,j\right\}}$ of $\bar{\mathsf{K}}$ for
any $1 \leq i < j \leq r$, whose relative interior is a $(\dim
\mathbb{W}_{ij} + 2)$-dimensional homogeneous convex cone of rank $2$, has
a spectrahedral representation of order $2$, and hence dimension at most
$3$. We remind the reader that this does not mean that $\mathbb{S}^{m}_{+}
\cap \mathbb{W}$ is a sparse spectrahedral cone because the $n_{i}$'s may
not all be $1$.

As an aside, when we combine the preceding discussion with our
characterization of selfdual homogeneous cones in
Theorem~\ref{theorem:selfdualhomogeneouscone}, we arrive at the
\begin{corollary}\label{theorem:dualsparsespectrahedralcone}
The rank of a homogeneous sparse spectrahedral cone matches the
Carath\'eodory number of its dual cone if and only if all its
indecomposable components are cones of positive definite real symmetric
matrices.
\end{corollary}
\begin{proof}
Theorem~\ref{theorem:sparsespectrahedralcone} states that a homogeneous
sparse spectrahedral cone has a rank-matching Carath\'eodory number. So, by
Theorem~\ref{theorem:selfdualhomogeneouscone}, its dual cone has a
rank-matching Carath\'eodory number if and only if it is selfdual. The
preceding discussion asserts that every Ishi spectrahedral representation
$\mathbb{S}^{n}_{++} \cap \mathbb{W}$ of a homogeneous sparse spectrahedral
cone has linear subspaces $\mathbb{W}_{ij}$ of dimensions at most one, and
the only indecomposable selfdual homogeneous convex cones with such Ishi
spectrahedral representations are the cones of positive definite real
symmetric matrices; see \cite[\S{II}.2]{Vin65}.
\end{proof}

We now return our attention to indecomposable sparse spectrahedral cones of
order $4$ that are also homogeneous convex cones. Such a homogeneous convex
cone must have rank $4$ and Ishi spectrahedral representation $\mathsf{K} =
\mathbb{S}^{n}_{++} \cap \mathbb{W}$ for some $n \geq 4$, where for all $1
\leq i < j \leq 4$, $\dim \mathbb{W}_{ij} \leq 1$ and $\dim \mathbb{W}_{ik}
= \dim \mathbb{W}_{jk}$ for all $k \in \left\{j+1, \dotsc, 4\right\}$
whenever $\mathbb{W}_{ij}$ is not the trivial subspace. Since $\mathsf{K}$
is indecomposable, at least one of $\mathbb{W}_{12}$, $\mathbb{W}_{13}$,
and $\mathbb{W}_{14}$ is nontrivial.
\begin{enumerate}[label=(\roman*)]
\item If only $\mathbb{W}_{12}$ is nontrivial, then $0 = \dim
\mathbb{W}_{1k} = \dim \mathbb{W}_{2k}$ for $k = 3, 4$; but this means that
$\mathbb{W}$, and hence $\mathsf{K}$, is decomposable as $\mathbb{W} =
\mathbb{W}_{\left\{1,2\right\}} \oplus \mathbb{W}_{\left\{3,4\right\}}$.

\item If only $\mathbb{W}_{13}$ is nontrivial, then $\dim \mathbb{W}_{12} =
0 = \dim \mathbb{W}_{14} = \dim \mathbb{W}_{34}$, so that the
indecomposability of $\mathbb{W}$ into, respectively,
$\mathbb{W}_{\left\{1,3\right\}} \oplus \mathbb{W}_{\left\{2,4\right\}}$
and $\mathbb{W}_{\left\{1,2,3\right\}} \oplus
\mathbb{W}_{\left\{4\right\}}$ means that $\mathbb{W}_{23}$ and
$\mathbb{W}_{24}$ are nontrivial; but this means that $\mathbb{W}_{24}$ and
$\mathbb{W}_{34}$ have different dimensions, and yet $\mathbb{W}_{23}$ is
nontrivial.

\item If only $\mathbb{W}_{14}$ is nontrivial, then $\dim \mathbb{W}_{12} =
\dim \mathbb{W}_{13} = 0$, so that the indecomposability of $\mathbb{W}$
means that at most one of $\mathbb{W}_{23}$, $\mathbb{W}_{24}$, and
$\mathbb{W}_{34}$ can be the trivial subspace. When $\mathbb{W}_{23}$ is
the trivial subspace, the homogeneous convex cone has the Ishi
spectrahedral representation
\[
\mathbb{S}^{4}_{++} \cap \left\{
\begin{pmatrix}
x_{1} & 0 & 0 & x_{5}
\\
0 & x_{2} & 0 & x_{6}
\\
0 & 0 & x_{3} & x_{7}
\\
x_{5} & x_{6} & x_{7} & x_{4}
\end{pmatrix}
\ \middle|\ x \in \mathbb{R}^{7}\right\}.
\]
Otherwise, the nontriviality of $\mathbb{W}_{23}$ means that
$\mathbb{W}_{24}$ and $\mathbb{W}_{34}$ have the same dimensions and are,
therefore, both nontrivial because at most one of them can be the trivial
subspace. This means that the homogeneous convex cone has the Ishi
spectrahedral representation
\[
\mathbb{S}^{4}_{++} \cap \left\{
\begin{pmatrix}
x_{1} & 0 & 0 & x_{6}
\\
0 & x_{2} & x_{5} & x_{7}
\\
0 & x_{5} & x_{3} & x_{8}
\\
x_{6} & x_{7} & x_{8} & x_{4}
\end{pmatrix}
\ \middle|\ x \in \mathbb{R}^{8}\right\}.
\]

\item If only $\mathbb{W}_{12}$ and $\mathbb{W}_{13}$ are nontrivial, then
$0 = \dim \mathbb{W}_{14} = \dim \mathbb{W}_{j4}$ for $j = 2, 3$; but this
means that $\mathbb{W}$ is decomposable as $\mathbb{W} =
\mathbb{W}_{\left\{1,2,3\right\}} \oplus \mathbb{W}_{\left\{4\right\}}$.

\item If only $\mathbb{W}_{12}$ and $\mathbb{W}_{14}$ are nontrivial, then
$0 = \dim \mathbb{W}_{13} = \dim \mathbb{W}_{23}$ and $1 = \dim
\mathbb{W}_{14} = \dim \mathbb{W}_{24}$, so that the indecomposability of
$\mathbb{W}$ into $\mathbb{W}_{\left\{1,2,4\right\}} \oplus
\mathbb{W}_{\left\{3\right\}}$ means that $\mathbb{W}_{34}$ is nontrivial.
This means that the homogeneous convex cone has the above Ishi
spectrahedral representation of dimension $8$.

\item If only $\mathbb{W}_{13}$ and $\mathbb{W}_{14}$ are nontrivial, then
$1 = \dim \mathbb{W}_{14} = \dim \mathbb{W}_{34}$. The indecomposability of
$\mathbb{W}$ into $\mathbb{W}_{\left\{1,3,4\right\}} \oplus
\mathbb{W}_{\left\{2\right\}}$ means that at least one of $\mathbb{W}_{23}$
and $\mathbb{W}_{24}$ is nontrivial. If $\mathbb{W}_{23}$ is nontrivial,
then $\dim \mathbb{W}_{24} = \dim \mathbb{W}_{34} = 1$ means that the
homogeneous convex cone has the Ishi spectrahedral representation
\[
\mathbb{S}^{4}_{++} \cap \left\{
\begin{pmatrix}
x_{1} & 0 & x_{5} & x_{7}
\\
0 & x_{2} & x_{6} & x_{8}
\\
x_{5} & x_{6} & x_{3} & x_{9}
\\
x_{7} & x_{8} & x_{9} & x_{4}
\end{pmatrix}
\ \middle|\ x \in \mathbb{R}^{9}\right\}.
\]
Otherwise, $\mathbb{W}_{24}$ must be nontrivial and hence the homogeneous
convex cone has the above Ishi spectrahedral representation of dimension
$8$.

\item If $\mathbb{W}_{12}$, $\mathbb{W}_{13}$, and $\mathbb{W}_{14}$ are
all nontrivial, then $1 = \dim \mathbb{W}_{13} = \dim \mathbb{W}_{23}$ and
$1 = \dim \mathbb{W}_{14} = \dim \mathbb{W}_{j4}$ for $j = 2, 3$ means that
the homogeneous convex cone is linearly isomorphic to the positive definite
cone $\mathbb{S}^{4}_{++}$.

\end{enumerate}
In summary, there are four linear isomorphism classes of indecomposable
homogeneous sparse spectrahedral cones of rank $4$, and they have
dimensions $7$, $8$, $9$, and $10$, respectively. We shall denote the above
representative cones of the first three classes by $\mathsf{K}^{4}(7)$,
$\mathsf{K}^{4}(8)$, and $\mathsf{K}^{4}(9)$, respectively.

On the other hand, the two sparse spectrahedral cones described by $P_{4}$
and $C_{4}$ are, respectively,
\[
\mathsf{K}(P_{4}) = \mathbb{S}^{4}_{++} \cap \left\{
\begin{pmatrix}
x_{1} & x_{5} & 0 & x_{6}
\\
x_{5} & x_{2} & 0 & 0
\\
0 & 0 & x_{3} & x_{7}
\\
x_{6} & 0 & x_{7} & x_{4}
\end{pmatrix}
\ \middle|\ x \in \mathbb{R}^{7}\right\}
\]
and
\[
\mathsf{K}(C_{4}) = \mathbb{S}^{4}_{++} \cap \left\{
\begin{pmatrix}
x_{1} & x_{5} & 0 & x_{7}
\\
x_{5} & x_{2} & x_{6} & 0
\\
0 & x_{6} & x_{3} & x_{8}
\\
x_{7} & 0 & x_{8} & x_{4}
\end{pmatrix}
\ \middle|\ x \in \mathbb{R}^{8}\right\}.
\]
These are not linearly isomorphic to the above homogeneous convex cones of
the respective dimensions for the following respective reasons.
\begin{enumerate}[label=(\roman*)]
\item The maximal proper face
\[
\left\{
\begin{pmatrix}
X & O_{3 \times 1}
\\
O_{3 \times 1}^{\top} & 0
\end{pmatrix}
\ \middle|\ X \in \mathbb{S}^{3}_{+}\right\}
\]
of $\mathbb{S}^{4}_{+}$ intersects with the linear span $\mathbb{W}(P_{4})$
of $\mathsf{K}(P_{4})$ to give the $4$-dimensional maximal proper face
\[
\mathbb{S}^{4}_{+} \cap \left\{
\begin{pmatrix}
X & O_{2 \times 1} & O_{2 \times 1}
\\
O_{2 \times 1}^{\top} & x & 0
\\
O_{2 \times 1}^{\top} & 0 & 0
\end{pmatrix}
\ \middle|\ X \in \mathbb{S}^{2}, \ x \in \mathbb{R}\right\}
\]
of the closure of $\mathsf{K}(P_{4})$. On the other hand, maximal proper
faces of the closure of the Ishi spectrahedral representation
$\mathsf{K}^{4}(7)$ are linearly isomorphic to the faces
$\mathsf{F}_{I_{4},\left\{1, 2, 3\right\}}$, $\mathsf{F}_{I_{4},\left\{1,
2, 4\right\}}$, $\mathsf{F}_{I_{4},\left\{1, 3, 4\right\}}$, or
$\mathsf{F}_{I_{4},\left\{2, 3, 4\right\}}$, which have dimensions $3$ or
$5$.

\item Extreme rays of the closure of the Ishi spectrahedral representation
$\mathsf{K}^{4}(8)$ are of one of the four forms
\[
\mathbb{R}_+
\begin{pmatrix}
1 \\ 0 \\ 0 \\ u_{6}
\end{pmatrix}
\begin{pmatrix}
1 \\ 0 \\ 0 \\ u_{6}
\end{pmatrix}^{\top},
\mathbb{R}_+
\begin{pmatrix}
0 \\ 1 \\ u_{5} \\ u_{7}
\end{pmatrix}
\begin{pmatrix}
0 \\ 1 \\ u_{5} \\ u_{7}
\end{pmatrix}^{\top},
\mathbb{R}_+
\begin{pmatrix}
0 \\ 0 \\ 1 \\ u_{8}
\end{pmatrix}
\begin{pmatrix}
0 \\ 0 \\ 1 \\ u_{8}
\end{pmatrix}^{\top},
\mathbb{R}_+
\begin{pmatrix}
0 \\ 0 \\ 0 \\ 1
\end{pmatrix}
\begin{pmatrix}
0 \\ 0 \\ 0 \\ 1
\end{pmatrix}^{\top}.
\]
So, all extreme rays are in the union of the two proper faces
\[
\mathbb{S}^{4}_{+} \cap \left\{
\begin{pmatrix}
X_{11} & O_{1 \times 2} & X_{12}
\\
O_{1 \times 2}^{\top} & O_{2 \times 2} & O_{2 \times 1}
\\
X_{12} & O_{2 \times 1}^{\top} & X_{22}
\end{pmatrix}
\ \middle|\ X \in \mathbb{S}^{2}\right\} \quad \text{and} \quad
\mathbb{S}^{4}_{+} \cap \left\{
\begin{pmatrix}
0 & O_{1 \times 3}
\\
O_{1 \times 3}^{\top} & X
\end{pmatrix}
\ \middle|\ X \in \mathbb{S}^{3}\right\},
\]
which are linearly isomorphic to the respective positive semidefinite cones
$\mathbb{S}^{2}_{+}$ and $\mathbb{S}^{3}_{+}$.

On the other hand, we deduce from
Proposition~\ref{proposition:facesofslices} that extreme rays of the
closure of the sparse spectrahedral cone $\mathsf{K}(C_{4})$ are extreme
rays of $\mathbb{S}^{4}_{+}$ that lie in the linear span
$\mathbb{W}(C_{4})$ of $\mathsf{K}(C_{4})$, and hence $\mathsf{F}(i) :=
\mathbb{R}_+ (T(i)^{\top} T(i))$ for $i = 1, 2, 3$, where
\[
\begin{aligned}
T(1) &=
\begin{pmatrix}
1 & 1 & 0 & 0
\end{pmatrix},
\\
T(2) &=
\begin{pmatrix}
0 & 1 & 1 & 0
\end{pmatrix},
\\
T(3) &=
\begin{pmatrix}
0 & 0 & 1 & 1
\end{pmatrix},
\end{aligned}
\]
are three particular extreme rays of the closure of $\mathsf{K}(C_{4})$.
The smallest face of the closure of $\mathsf{K}(C_{4})$ containing all
three of them is
\[
\begin{split}
&\left\{
\begin{pmatrix}
T(1) \\ T(2) \\
T(3)
\end{pmatrix}^{\top}
X
\begin{pmatrix}
T(1) \\ T(2) \\
T(3)
\end{pmatrix}
\ \middle|\ X \in \mathbb{S}^{3}\right\} \cap \left\{
\begin{pmatrix}
x_{1} & x_{5} & 0 & x_{7}
\\
x_{5} & x_{2} & x_{6} & 0
\\
0 & x_{6} & x_{3} & x_{8}
\\
x_{7} & 0 & x_{8} & x_{4}
\end{pmatrix}
\ \middle|\ x \in \mathbb{R}^{8}\right\}
\\
&= \mathbb{S}^{4}_{+} \cap \left\{
\begin{pmatrix}
T(1) \\ T(2) \\
T(3)
\end{pmatrix}^{\top}
X
\begin{pmatrix}
T(1) \\ T(2) \\
T(3)
\end{pmatrix}
\ \middle|\ X \in \mathbb{S}^{3}, \ X_{12} + X_{13} = X_{13} + X_{23} =
0\right\},
\end{split}
\]
whose relative interior is linearly isomorphic to a four-dimensional
spectrahedral cone of order $3$, while the smallest face of the closure of
$\mathsf{K}(C_{4})$ containing each pair $(\mathsf{F}(i), \mathsf{F}(j))$
of extreme rays with $1 \leq i < j \leq 3$ is
\[
\begin{split}
&\left\{
\begin{pmatrix}
T(i) \\ T(j)
\end{pmatrix}^{\top}
X
\begin{pmatrix}
T(i) \\ T(j)
\end{pmatrix}
\ \middle|\ X \in \mathbb{S}^{2}\right\} \cap \left\{
\begin{pmatrix}
x_{1} & x_{5} & 0 & x_{7}
\\
x_{5} & x_{2} & x_{6} & 0
\\
0 & x_{6} & x_{3} & x_{8}
\\
x_{7} & 0 & x_{8} & x_{4}
\end{pmatrix}
\ \middle|\ x \in \mathbb{R}^{8}\right\}
\\
&= \mathbb{S}^{4}_{+} \cap \left\{x_{1} T(i)^{\top} T(i) + x_{2}
T(j)^{\top} T(j)\ \middle|\ x \in \mathbb{R}^{2}\right\},
\end{split}
\]
which is linearly isomorphic to $\mathbb{R}_+^{2}$. Therefore, any face of
the closure of $\mathsf{K}(C_{4})$ containing at least two of these three
extreme rays cannot be contained in a face that is linearly isomorphic to
either $\mathbb{S}^{2}_{+}$ or $\mathbb{S}^{3}_{+}$, because neither
$\mathbb{S}^{2}_{+}$ nor $\mathbb{S}^{3}_{+}$ has a face that is linearly
isomorphic to either of the above smallest containing faces. This means
that these three extreme rays cannot be contained in the union of two faces
of the closure of $\mathsf{K}(C_{4})$ that are respectively linearly
isomorphic to $\mathbb{S}^{2}_{+}$ and $\mathbb{S}^{3}_{+}$.

\end{enumerate}

Thus, we have a new geometric proof of the previously known result that
\begin{theorem}\label{theorem:homogeneouschordalgraph}
A sparse spectrahedral cone is a homogeneous convex cone if and only if it
is described by a homogeneous chordal graph.
\end{theorem}

\section{Acknowledgement}

The author expresses his gratitude to two anonymous referees for their
comments and suggestions that led to the improved presentation of this
manuscript, particularly in Sections 4.1 and 5.

\section*{Declarations}

\subsection*{Funding}
No funding was received to assist with the preparation of this manuscript.

\subsection*{Competing interests}
The author has no competing interests to declare that are relevant to the
content of this article.

\end{document}